\documentclass[11pt]{article}
\usepackage[english]{babel}
\usepackage[utf8]{inputenc}
\usepackage{fullpage}
\usepackage{enumerate}

\usepackage{amsmath}
\usepackage{amsfonts}
\usepackage{amsthm}
\usepackage{mathtools}
\usepackage{dsfont}
\usepackage{graphicx}
\usepackage[colorlinks=false,linkcolor=black]{hyperref}

\usepackage{color}
\definecolor{grey}{rgb}{.7,.7,.7}

\usepackage[hang]{subfigure}

\newcommand{\R}{\mathbb R}
\newcommand{\N}{\mathbb N}

\newcommand{\cB}{\mathcal{B}}
\newcommand{\cC}{\mathcal{C}}

\newcommand{\cF}{\mathcal{F}}
\newcommand{\cG}{\mathcal{G}}
\newcommand{\cH}{\mathcal{H}}

\newcommand{\cL}{\mathcal{L}}
\newcommand{\cM}{\mathcal{M}}
\newcommand{\cP}{\mathcal{P}}

\newcommand{\e}{\varepsilon}
\newcommand{\g}{\gamma}
\renewcommand{\a}{\alpha}
\newcommand{\hmu}{{\hat \mu}}

\def\card{{\rm card}}

\newtheorem{theo}{Theorem}[section]
\newtheorem{prop}[theo]{Proposition}
\newtheorem{lemma}[theo]{Lemma}
\newtheorem{dfn}{Definition}[section]
\newtheorem{corollary}[theo]{Corollary}
\newtheorem{hypothesis}{Hypothesis}[section]

\theoremstyle{remark}
\newtheorem{question}{Question}
\newtheorem{remk}{Remark}[section]

\newtheorem{xmpl}[theo]{Example}

\DeclareMathOperator*{\dist}{dist}

\DeclareMathOperator*{\diam}{diam}

\makeatletter
\newcommand{\subjclass}[2][2010]{%
  \let\@oldtitle\@title%
  \gdef\@title{\@oldtitle\footnotetext{#1 \emph{Mathematics Subject Classification:} #2}}%
}
\newcommand{\keywords}[1]{%
  \let\@@oldtitle\@title%
  \gdef\@title{\@@oldtitle\footnotetext{\emph{Key words and phrases.} #1.}}%
}
\makeatother

\title{Recovering measures from approximate values on balls}
\author{Blanche \textsc{Buet}, Gian Paolo \textsc{Leonardi}}

\begin{document}
     \maketitle

\renewcommand{\phi}{\varphi}
\renewcommand{\epsilon}{\varepsilon}
\newcommand{\one}{\mathds{1}}

\begin{abstract}
In a metric space $(X,d)$ we reconstruct an approximation of a Borel measure $\mu$ starting from a premeasure $q$ defined on the collection of closed balls, and such that $q$ approximates the values of $\mu$ on these balls. More precisely, under a geometric assumption on the distance ensuring a Besicovitch covering property, and provided that there exists a Borel measure on $X$ satisfying an asymptotic doubling-type condition, we show that a suitable packing construction produces a measure $\hmu^{q}$ which is equivalent to $\mu$. Moreover we show the stability of this process with respect to the accuracy of the initial approximation. We also investigate the case of signed measures.
\end{abstract}

	\tableofcontents

\section{Introduction}

Is a Borel measure $\mu$ on a metric space $(X,d)$ fully determined by its values on balls? In the context of general measure theory, such a question appears to be of extremely basic nature. The answer (when it is known)  strongly depends upon the interplay between the measure and the metric space. A clear overview on the subject is given in \cite{handbook_geometry_banach_spaces}. Let us mention some known facts about this issue. When $X=\R^{n}$ equipped with a norm, the answer to the above question is in the affirmative. The reason is the following: if two Radon measures $\mu$ and $\nu$ coincide on every ball $B_{r}(x)\subset \R^{n}$, then in particular they are mutually absolutely continuous, therefore by the Radon-Nikodym-Lebesgue Differentiation Theorem one has $\mu(A) = \int_{A}\eta\, d\nu = \nu(A)$ for any Borel set $A$, where
\[
\eta(x) = \lim_{r \to 0} \frac{\mu (B_r (x))}{\nu (B_r(x))} = 1
\]
is the Radon-Nikodym derivative of $\mu$ with respect to $\nu$ (defined for $\nu$-almost all $x\in \R^{n}$).  
More generally, the same fact can be shown for any pair of Borel measures on a finite-dimensional Banach space $X$. 
Unfortunately, the Differentiation Theorem is valid on a Banach space $X$ if and only if $X$ is finite-dimensional. Of course, this does not prevent in general the possibility that Borel measures are uniquely determined by their values on balls. Indeed, Preiss and Ti$\rm{\check{s}er}$ proved in \cite{preiss_tiser} that in separable Banach spaces, two finite Borel measures coinciding on all balls also coincide on all Borel sets. Nevertheless, if this coincidence turns out to be satisfied only on balls of radius, say, less than $1$, then the question still stands. In the case of separable metric spaces, Federer introduced in \cite{federer} a geometrical condition on the distance (see Definition \ref{def:federer}) implying a Besicovitch-type covering lemma that can be used to show the property above, i.e., that any finite Borel measure is uniquely identified by its values on closed balls. When this condition on the distance is dropped, some examples of metric spaces and of pairs of distinct Borel measures coinciding on balls of upper-bounded diameter are known (see \cite{davies}). 

Here we consider the case of a separable metric space $(X,d)$ where Besicovitch covering lemma (or at least some generalized version of it) holds,  and we ask the following questions:
\begin{question}\label{question_1}\it
How can we reconstruct a Borel measure from its values on balls, and especially, what about the case of signed measures? 
\end{question}
A classical approach to construct a measure from a given \emph{premeasure} $p$ defined on a family $\cC$ of subsets of $X$ (here the premeasure $p$ is defined on closed balls) is to apply Carath\'eodory constructions (Method I and Method II, see \cite{bruckner_thomson}) to obtain an outer measure. 
We recall that a premeasure $p$ is a nonnegative function, defined on a given family $\cC$ of subsets of $X$, such that $\emptyset \in \cC$ and $p(\emptyset)=0$. By Method I, an outer measure $\mu^\ast$ is defined starting from $p$ as
\[
\mu^\ast (A) = \inf \left\lbrace \sum_{k=1}^\infty p(B_k) \: : \:  B_k \in \cC \text{ and } A \subset \bigcup_{k=1}^\infty B_k  \right\rbrace \: ,
\] 
for any $A \subset X$. But, as it is explained in \cite{bruckner_thomson} (Section $3.2$), Method I does not take into account that $X$ is a metric space, thus the resulting outer measure can be incompatible with the metric on $X$, in the sense that open sets are not necessarily $\mu^\ast$-measurable. On the other hand, Method II is used to define Hausdorff measures (see Theorem~\ref{thm_caratheodory_construction}) and it always produces a metric outer measure $\mu^{\ast}$, for which Borel sets are $\mu^\ast$-measurable.

As for a signed measure $\mu = \mu^+ - \mu^-$, the main problem is that, given a closed ball $B$, it is impossible to directly reconstruct $\mu^+ (B)$ and $\mu^- (B)$ from $\mu(B)$. The idea is, then, to apply Carath\'eodory's construction to the premeasure $p^+ (B) = \left( \mu(B) \right)_+$ (here $a_{+}$ denotes the positive part of $a\in \R$) and check that the resulting outer measure is actually $\mu^+$. Then, by a similar argument we recover $\mu^{-}$. 

Now we consider the problem of reconstructing a measure $\mu$ from approximate values on balls. We thus assume that a premeasure $q$, defined on closed balls, is given and satisfies the following two properties: for some $0<\a\le 1$, $C\ge 1$, and $r_{0}>0$,  
\begin{equation} \label{eq_mean_premeasure}
\begin{split}
(i)&\ \ q(B_{r}(x)) \ge C^{-1}\mu(B_{\a r}(x)) \,,\\ 
(ii)&\ \ q(B_{r}(x)) \le C \mu(B_{r}(x))
\end{split}
\end{equation} 
holds for all $0<r<r_{0}$ and all $x\in X$. 

\begin{question} \label{question_2} \it
Given a positive Borel measure $\mu$ and a premeasure $q$ defined on balls and satisfying the two conditions in \eqref{eq_mean_premeasure}, is it possible to reconstruct an approximation up to constants of $\mu$ from $q$? What about the case when $\mu$ is a signed measure?
\end{question}

First notice that under the assumptions \eqref{eq_mean_premeasure} the best one can obtain is a reconstruction $\hmu$ of $\mu$ such that 
\[
\hat C^{-1}\mu \le \hmu \le \hat C\, \mu
\]
for some constant $\hat C\ge 1$. We stress that, in the case $\a = 1$ in \eqref{eq_mean_premeasure}(i), this can be easily obtained via Carath\'eodory Method II (with $\hat C = C$) while in the case $0<\a<1$ Carath\'eodory construction does not provide in general such a measure $\hmu$. 
%
%
Loosely speaking, a loss of mass can happen in the recovery process, as the example presented in section \ref{section:2.1} shows.

In order to recover $\mu$, or at least some measure equivalent or comparable to $\mu$, the choice of the centers of the balls in the collection, which are used to cover the support of $\mu$, is crucial. Indeed they must be placed in some nearly-optimal positions, such that even the concentric balls with smaller radius have a significant overlapping with the support of $\mu$. This led us to consider a packing-type construction. The packing construction is mainly used to define the packing $s$-dimensional measure and its associated notion of packing dimension. It is in some sense dual to the construction leading to Hausdorff measure and dimension, and was introduced by C. Tricot in \cite{tricot_0}. Then Tricot and Taylor in \cite{tricot} extended it to the case of a general premeasure. In our setting we show that this packing construction can be formulated in a simpler and more manageable way (see section \ref{sect:packtype}).

Since the lower bound on $q(B_{r}(x))$ is given in terms of $\mu(B_{\a r}(x))$, in the case $0<\a<1$ and under some additional assumptions on the metric space $(X,d)$ we prove a suitable version of Besicovitch covering lemma (see Proposition \ref{prop_diffusion_in_X} and Corollary \ref{bescivoic_with_doubling_balls}), which represents a key ingredient in our construction and seems also of independent interest.

Some further explanations about the assumption \eqref{eq_mean_premeasure} on $q(B_{r}(x))$ are in order. An example of $q(B_{r}(x))$ satisfying \eqref{eq_mean_premeasure} is
\begin{equation}\label{varif-prem}
q(B_{r}(x)) = \frac 1r \int_{0}^{r} \mu(B_{s}(x))\, ds\,;
\end{equation}
more generally one could consider 
\[
q(B_{r}(x)) = \frac 1r \int_{0}^{r} \mu(B_{s}(x))\,\omega(s/r) ds
\]
where $\omega:(0,1)\to (0,+\infty)$ is a non-increasing probability density function. Notice that in $\R^{n}$ this last expression corresponds to the convolution of $\mu$ with $\omega(|x|/r)$, while for a general metric space it may be understood as an extension of the convolution operation.

We also remark that the starting motivation of our study is related to the problem of rectifiability of a $d$--varifold $V$ in $\R^{n}$ obtained as the limit of ``discrete varifolds'' (see \cite{Blanche-thesis,Blanche-rectifiability,BuetLeonardiMasnou}).

The paper is organized as follows. In Section $2$, we explain how to reconstruct a positive measure and then a signed measure (Theorem~\ref{prop_positive_case}) from their values on balls, thanks to Carath\'eodory's construction, answering Question~\ref{question_1}. Section $3$ deals with  Question~\ref{question_2}, that is, the reconstruction of a measure starting from a premeasure satisfying \eqref{eq_mean_premeasure}. After explaining the limitations of Carath\'eodory's construction for this problem, we prove our main result, Theorem~\ref{thm_main}, saying that by suitable packing constructions one can reconstruct a signed measure equivalent to the initial one in any metric space $(X,d)$ which is directionally limited and endowed with an asymptotically doubling measure $\nu$ (see Hypothesis \ref{hypo1} in page \pageref{hypo1}).

\subsection*{Some notations}

Let $(X,d)$ be a metric space.
\begin{itemize}
\item $\cB(X)$ denotes the $\sigma$--algebra of Borel subsets of $X$.
\item $B_r (x) = \left\lbrace y \in X \: | \: d(y,x) \leq r \right\rbrace$ is the closed ball of radius $r>0$ and center $x \in X$.
\item $U_r (x) = \left\lbrace y \in X \: | \: d(y,x) < r \right\rbrace$ is the open ball of radius $r>0$ and center $x \in X$.
\item $\cC$ denotes the collection of closed balls of $X$ and for $\delta > 0$, $\cC_\delta$ denotes the collection of closed balls of diameter $\leq \delta$.
\item $\cL^n$ is the Lebesgue measure in $\R^n$.
\item $\cP (X)$ is the set of all subsets of $X$.
\item $\card A$ is the cardinality of the set $A$.
\item $\N^\ast = \{ 1 , 2 , \ldots \}$.
\end{itemize}

\section{Carath\'eodory metric construction of outer measures}
We recall here some standard definitions and well-known facts about general measures, focusing in particular on the construction of measures from premeasures, in the sense of Carath\'eodory Method II \cite{bruckner_thomson}.

\subsection{Outer measures and metric outer measures}

\begin{dfn}[Outer measure]
Let $X$ be a set, and let $\mu^\ast : \cP (X) \rightarrow [0;+\infty]$ satisfying
\begin{enumerate}[$(i)$]
\item $\mu^\ast (\emptyset ) = 0$.
\item $\mu^\ast$ is monotone: if $A \subset B \subset X$, then $\mu^\ast (A) \leq \mu^\ast (B)$.
\item $\mu^\ast$ is \emph{countably subadditive}: if $(A_k)_{k \in \N}$ is a sequence of subsets of $X$, then
\[
\mu^\ast \left( \bigcup_{k=1}^\infty A_k \right) \leq \sum_{k=1}^\infty \mu^\ast (A_k) \: .
\] 
\end{enumerate}
Then $\mu^\ast$ is called an \emph{outer measure} on $X$.
\end{dfn}

In order to obtain a measure from an outer measure, one defines the measurable sets with respect to $\mu^\ast$.

\begin{dfn}[$\mu^\ast$--measurable set]
Let $\mu^\ast$ be an outer measure on $X$. A set $A \subset X$ is \emph{$\mu^\ast$--measurable} if for all sets $E \subset X$,
\[
\mu^\ast (E) = \mu^\ast (E \cap A) + \mu^\ast (E \setminus A) \: .
\]
\end{dfn}

We can now define a measure associated with an outer measure. Thanks to the definition of $\mu^\ast$--measurable sets, the additivity of $\mu^\ast$ among the measurable sets is straightforward, actually it happens that $\mu^\ast$ is $\sigma$--additive on $\mu^\ast$--measurable sets.

\begin{theo}[Measure associated with an outer measure, see Theorem $2.32$ in \cite{bruckner_thomson}] \label{thm_from_outer_measure_to_measure}
Let $X$ be a set, $\mu^\ast$ an outer measure on $X$, and $\cM$ the class of $\mu^\ast$--measurable sets. Then $\cM$ is a $\sigma$--algebra and $\mu^\ast$ is countably additive on $\cM$. Thus the set function $\mu$ defined on $\cM$ by $\mu(A) = \mu^\ast (A)$ for all $A \in \cM$ is a measure.
\end{theo}

We now introduce \emph{metric outer measures}.

\begin{dfn}
Let $(X, d)$ be a metric space and $\mu^\ast$ be an outer measure on $X$. $\mu^\ast$ is called a \emph{metric outer measure} if
\[
\nu(E \cup F) = \nu(E) + \nu(F)
\]
for any $E,F \subset X$ such that $d(E,F) >0$.
\end{dfn}

When $\mu^\ast$ is a metric outer measure, every Borel set is $\mu^\ast$--measurable and thus the measure $\mu$ associated with $\mu^\ast$ is a Borel measure.

\begin{theo}[Carath\'eodory's Criterion, see Theorem $3.8$ in \cite{bruckner_thomson}]
Let $\mu^\ast$ be an outer measure on a metric space $(X,d)$. Then every Borel set in $X$ is $\mu^{\ast}$-measurable if and only if $\mu^\ast$ is a metric outer measure. In particular, a metric outer measure is a Borel measure.
\end{theo}

We recall two approximation properties of Borel measures defined on metric spaces.

\begin{theo}[see Theorems $3.13$ and $3.14$ in \cite{bruckner_thomson}] \label{thm_approximation_Borel_measure}
Let $(X,d)$ be a metric space and $\mu$ be a Borel measure on $X$.
\begin{itemize}
\item \emph{Approximation from inside}: Let $F$ be a Borel set such that $\mu (F) < +\infty$, then for any $\epsilon > 0$, there exists a closed set $F_\epsilon \subset F$ such that $\mu (F \setminus F_\epsilon) < \epsilon$.
\item \emph{Approximation from outside}: Assume that $\mu$ is finite on bounded sets and let $F$ be a Borel set, then
\[
\mu (F) = \inf \{ \mu (U) \: : \: F \subset U, \, U \text{ open set} \} \: .
\]
\end{itemize}
\end{theo}

We can now introduce Carath\'eodory's construction of metric outer measures (Method II, see \cite{bruckner_thomson}).

\begin{dfn}[Premeasure]
Let $X$ be a set and $\cC$ be a family of subsets of $X$ such that $\emptyset \in \cC$. A nonnegative function $p$ defined on $\cC$ and such that $p(\emptyset) = 0$ is called a \emph{premeasure}.
\end{dfn}


\begin{theo}[Carath\'eodory's construction, Method II] \label{thm_caratheodory_construction}
Suppose $(X,d)$ is a metric space and $\cC$ is a family of subsets of $X$ which contains the empty set. Let $p$ be a premeasure on $\cC$. For each $\delta > 0$, let
\[
\cC_\delta= \{A \in \cC \: | \: \diam(A) \leq \delta\} 
\]
and for any $E \subset X$ define
\[
\nu^p_\delta(E) = \inf \left\lbrace \sum_{i=0}^\infty p(A_i)\,\bigg| \, E \subset \bigcup_{i\in \N} A_i,\forall i, \: A_i\in \cC_\delta \right\rbrace.
\]
As $\nu^p_\delta \geq \nu^p_{\delta^\prime}$ when $\delta \leq \delta^\prime$, the limit
\[
\nu^{p,\ast} (E) = \lim_{\delta \to 0} \nu^p_\delta(E)
\]
exists (possibly infinite). Then $\nu^{p,\ast}$ is a metric outer measure on $X$.
\end{theo}

\subsection{Effects of Carath\'eodory's construction on positive Borel measures}
Let $(X,d)$ be a metric space and $\mu$ be a positive Borel $\sigma$--finite measure on $X$. Let $\cC$ be the set of closed balls and let $p$ be the premeasure defined on $\cC$ by
\begin{equation} \label{dfn_premeasure} \begin{array}{lllll}
p & : & \cC & \rightarrow & [0,+\infty] \\
  &   &   B & \mapsto     &  \mu (B)
\end{array} \end{equation}
Let $\mu^{p,\ast}$ be the metric outer measure obtained by Carath\'eodory's metric construction applied to $(\cC,p)$ and then $\mu^p$ the Borel measure associated with $\mu^{p,\ast}$. Then, the following question arises. 
\begin{question} \label{question_positive_case}\it 
Do we have $\mu^p = \mu$? In other terms, can we recover the initial measure by Carath\'eodory's Method II?
\end{question}
The following lemma shows one of the two inequalities needed to positively answer Question \ref{question_positive_case}.
\begin{lemma} \label{lem_easy_inequality}
Let $(X,d)$ be a metric space and $\mu$ be a positive Borel measure on $X$. Then, in the same notations as above, we have $\mu \leq \mu^p$.
\end{lemma}

\begin{proof}
Let $A \subset X$ be a Borel set, we have to show that $\mu (A) \leq \mu^p (A) = \mu^{p,\ast} (A)$. This inequality relies only on the definition of $\mu^p_\delta$ as an infimum. Indeed, let $\delta >0$ be fixed, then for any $\eta > 0$ there exists a countable collection of closed balls $( B_j^\eta )_{j \in \N} \subset \cC_\delta$ such that 
\[
A \subset \bigcup_j B_j^\eta\qquad \text{and}\qquad \mu^p_\delta (A) \geq \sum_{j=1}^\infty p (B_j^\eta) - \eta\,,
\] 
so that
\[
\mu^p_\delta (A) + \eta \geq \sum_{j=1}^\infty p (B_j^\eta) = \sum_{j=1}^\infty \mu (B_j^\eta) \geq \mu \Big( \bigcup_j B_j^\eta  \Big) \geq \mu (A) \: .
\]
Letting $\eta \rightarrow 0$ and then $\delta \rightarrow 0$ leads to $\mu (A) \leq \mu^p (A)$.
\end{proof}

A consequence of Davies' result \cite{davies} is that the other inequality cannot be true in general. We need extra assumptions on the metric space $(X,d)$ ensuring that open sets are ``well approximated'' by closed balls, that is, we need some specific covering property. In $\R^n$ with the Euclidean norm, this approximation of open sets by disjoint unions of balls is provided by Besicovitch Theorem, which we recall here:

\begin{theo}[Besicovitch Theorem, see Corollary $1$ p. $35$ in \cite{evans}]
\label{thm_besicovitch}
Let $\mu$ be a Borel measure on $\R^n$ and consider any collection $\cF$ of non degenerated closed balls. Let $A$ denote the set of centers of the balls in $\cF$. Assume $\mu (A) < +\infty$ and that 
\[
\inf \left\lbrace r>0 \: | \: B_r (a) \in \cF \right\rbrace = 0 \qquad \forall\, a \in A\,.
\]
Then, for every open set $U \in \R^n$, there exists a countable collection $\cG$ of disjoint balls in $\cF$ such that
\[
\bigsqcup_{B \in \cG} B \subset U \quad \text{and} \quad \mu \left( (A \cap U) - \bigsqcup_{B \in \cG} B \right) = 0 \: .
\]
\end{theo}

A generalization of Besicovitch Theorem for metric spaces is due to Federer, under a geometric assumption involving the distance function.

\begin{dfn}[Directionally limited distance, see $2.8.9$ in \cite{federer}]
\label{def:federer}
Let $(X,d)$ be a metric space, $A \subset X$ and $\xi > 0$, $0<\eta\leq \frac{1}{3}$, $\zeta \in \N^\ast$. The distance $d$ is said to be \emph{directionally $(\xi,\eta,\zeta)$--limited at $A$} if the following holds. Take any $a \in A$ and $B \subset A \cap \left( U_\xi (a) \setminus \{a\} \right)$ satisfying the following property: let $b, c \in B$ with $b \neq c$ and assume without loss of generality that $d(a,b) \geq d(a,c)$, then for all $x \in X$ such that $d(a,x) = d(a,c)$ and $d(a,x) + d(x,b)  = d(a,b)$ one has
\begin{equation}\label{directionally_limited_distance_eq} 
\frac{d(x,c)}{d(a,c)} \geq \eta \: .
\end{equation}
Then $\card B \leq \zeta$.
\end{dfn}

Let us say a few words about this definition. If $(X,|\cdot|)$ is a Banach space with strictly convex norm, then the above relations involving $x$ imply that 
\[
x = a + \frac{|a-c|}{|a-b|} (b-a ) \: ,
\]
hence in this case \eqref{directionally_limited_distance_eq} is equivalent to
\[
\frac{d(x,c)}{d(a,c)} = \left| \frac{c-a}{|c-a|}  - \frac{b-a}{|b-a|} \right| \geq \eta \: .
\]
Consequently, if $X$ is finite-dimensional, and thanks to the compactness of the unit sphere, for a given $\eta$ there exists $\zeta \in \N$ such that $(X, |\cdot |)$ is directionally $(\xi,\eta,\zeta)$-limited for all $\xi>0$. Hereafter we provide two examples of metric spaces that are not directionally limited.

\begin{xmpl}
Consider in $\R^2$ the union $X$ of a countable number of half-lines, joining at the same point $a$. Then the geodesic metric $d$ induced on $X$ by the ambient metric is not directionally limited at $\{a\}$.

\noindent \begin{minipage}{0.70\textwidth}
Indeed let $B = X \cap \{y \: : \: d(a,y) = \xi \}$, let $b$ and $c \in B$ lying in two different lines, at the same distance $d(a,b) = d(a,c)=\xi$ of $a$. Then $x \in X$ such that $d(a,x) =d(a,c)=\xi$ and $d(b,x) = d(a,b) - d(a,c) = 0$ implies $x=b$ and thus
\[
\frac{d(x,c)}{d(a,c)} = \frac{d(b,c)}{\xi} = \frac{2\xi}{\xi} = 2 \: .
\]
but $\card \, B$ is not finite.
\end{minipage}
\quad
\begin{minipage}{0.25\textwidth}
\includegraphics[scale=0.8]{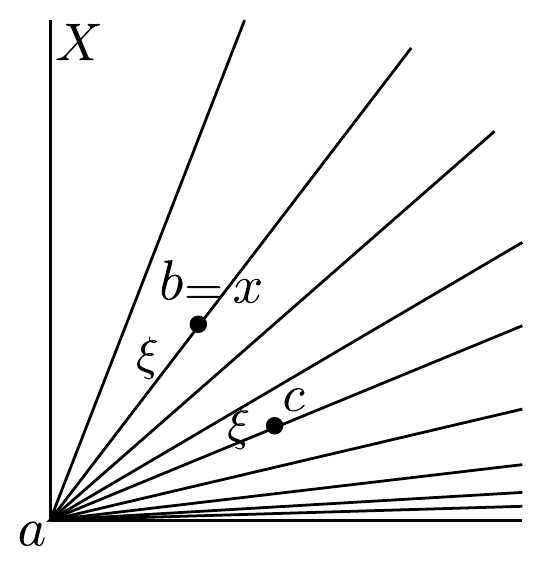}
\end{minipage}
\end{xmpl}

\begin{xmpl}
If $X$ is a separable Hilbert space and $B=(e_k)_{k \in \N}$ a Hilbert basis, $a \in H$ and $b= a + e_j, \, c= a + e_k \in a+B$, the Hilbert norm is strictly convex so that $d(a,x) = d(a,c)$, $d(b,x) = d(a,b) - d(a,c)$ uniquely define $x$ as
\[
x = a + \frac{|e_k|}{|e_j|} e_j = b \text{ and } \frac{d(x,c)}{d(a,c)} = \left| e_k  - e_j \right| = 2 \geq \eta
\]
for all $\eta \leq \frac{1}{3}$ and $\card (a + B)$ is infinite. Therefore $H$ is not directionally limited (nowhere).
\end{xmpl}

We can now state the generalized versions of Besicovitch Covering Lemma and Besicovitch Theorem for directionally limited metric spaces.

\begin{theo}[Generalized Besicovitch Covering Lemma, see $2.8.14$ in \cite{federer}] \label{thm_generalized_besicovic_federer_covering_lemma}
Let $(X,d)$ be a separable metric space directionally $(\xi,\eta,\zeta)$--limited at $A \subset X$. Let $0 < \delta < \frac{\xi}{2}$ and $\cF$ be a family of closed balls of radii less than $\delta$ such that each point of $A$ is the center of some ball of $\cF$. Then, there exists $2 \zeta +1$ countable subfamilies of $\cF$ of disjoint closed balls, $\cG_1, \ldots , \cG_{2 \zeta + 1}$ such that
\[
A \subset \bigcup_{j = 1}^{2 \zeta + 1} \bigsqcup_{B \in \cG_j} B \: .
\]
\end{theo}

\begin{remk}
In $\R^n$ endowed with the Euclidean norm it is possible to take $\xi = +\infty$ and $\zeta$ only dependent on $\eta$ and $n$. If we fix $\eta = \frac{1}{3}$, then $\zeta = \zeta_n$ only depends on the dimension $n$.
\end{remk}

\begin{theo}[Generalized Besicovitch Theorem, see $2.8.15$ in \cite{federer}] \label{thm_generalized_besicovic_federer}
Let $(X,d)$ be a separable metric space directionally $(\xi,\eta,\zeta)$--limited at $A \subset X$. Let $\cF$ be a family of closed balls of $X$ satisfying
\begin{equation} \label{hyp_fine_cover}
\inf \left\lbrace r>0 \: | \: B_r (a) \in \cF \right\rbrace = 0, \qquad \forall\, a\in A\:,
\end{equation}
and let $\mu$ be a positive Borel measure on $X$ such that $\mu (A) < +\infty$. Then, for any open set $U \subset X$ there exists a countable disjoint family $\cG$ of $\cF$ such that
\[
\bigsqcup_{B \in \cG} B \subset U \quad \text{and} \quad \mu \left( (A \cap U) - \bigsqcup_{B \in \cG} B \right) = 0 \: .
\]
\end{theo} 

We can now prove the coincidence of the initial measure and the reconstructed measure under assumptions of Theorem~\ref{thm_generalized_besicovic_federer}.

\begin{prop} \label{prop_positive_case}
Let $(X,d)$ be a separable metric space directionally $(\xi,\eta,\zeta)$--limited at $X$. Let $\mu$ be a positive Borel measure on $X$, finite on bounded sets. Let $\cC$ be the family of closed balls in $X$ and let $p$ be the premeasure defined on $\cC$ by \eqref{dfn_premeasure}. Denote by $\mu^{p,\ast}$ the metric outer measure obtained by Carath\'eodory's metric construction applied to $(\cC,p)$ and by $\mu^p$ the Borel measure associated with $\mu^{p,\ast}$. Then $\mu^p = \mu$.
\end{prop}

\begin{proof} 
\textbf{Step one.} We first show that $\mu^p$ is a Borel measure finite on bounded sets. First we 
recall that $\mu^{p,\ast}$ is a metric outer measure by Theorem~\ref{thm_caratheodory_construction}, then thanks to Theorem \ref{thm_from_outer_measure_to_measure} $\mu^p$ is a Borel measure. Let us prove that $\mu^p$ is finite on bounded sets. Fix $A \subset X$ a bounded Borel set and apply Besicovitch Covering Lemma (Theorem~\ref{thm_generalized_besicovic_federer_covering_lemma}) with the family
\[
\cF_\delta = \left\lbrace B = B_r(x) \text{ closed ball} \: : \: x \in A \text{ and } \diam B \leq \delta \right\rbrace \: ,
\]
to get $2 \zeta + 1$ countable subfamilies $\cG_1^\delta \ldots , \cG_{2\zeta +1}^\delta$ of disjoint balls in $\cF$ such that
\[
A \subset \bigcup_{j=1}^{2\zeta +1} \bigsqcup_{B \in \cG_j^\delta} B \: .
\]
Therefore,
\[
\mu_\delta^p (A) \leq \sum_{j=1}^{2 \zeta + 1} \sum_{B \in \cG_j^\delta} p(B) \leq \sum_{j=1}^{2 \zeta +1} \mu \left( \bigsqcup_{B \in \cG_j^\delta} B \right) \leq (2\zeta + 1) \mu (A + B_\delta (0)) \leq (2\zeta + 1) \mu (A + B_1 (0)) \: ,
\]
where $A + B_1 (0) = \bigcup_{x \in A} B_1 (x)$ is bounded, thus $ \mu (A + B_1 (0)) < +\infty$ by assumption and hence for all $0 < \delta < 1$
\[
\mu_\delta^p (A) \leq (2\zeta + 1) \mu (A + B_1 (0)) < +\infty \: ,
\]
whence $\mu^{p,\ast}(A) < +\infty$. The claim is proved since $\mu^p = \mu^{p, \ast}$ on Borel sets.

\textbf{Step two.} We now prove that for any open set $U \subset X$, it holds $ \mu^p (U) \leq \mu (U)$. Let $\delta > 0$ be fixed. The collection of closed balls 
\[
\cC_\delta = \left\lbrace B_r (x) \: | \: x \in U, \: 0< 2r \leq \delta \right\rbrace \: .
\]  
satisfies assumption \eqref{hyp_fine_cover}. We can apply Theorem~\ref{thm_generalized_besicovic_federer} to $\mu^p$ and get a countable collection $\cG^\delta$ of disjoint balls in $\cC_\delta$ such that
\[
\bigsqcup_{B \in \cG^\delta} B \subset U \quad \textrm{and} \quad \mu^p (U) = \mu^p \left( \bigsqcup_{B \in \cG^\delta} B \right) \: .
\]
At the same time we have
\begin{equation} \label{proof_positive_measure_case_1}
\mu^p_\delta \left( \bigsqcup_{B \in \cG^\delta} B \right) \leq \sum_{B \in \cG^\delta } p (B) = \sum_{B \in \cG^\delta} \mu (B) = \mu \left( \bigsqcup_{B \in \cG^\delta} B \right) \leq \mu (U) \: .
\end{equation}
We fix any decreasing infinitesimal sequence $(\delta_n)_{n \in \N}$ and define
$\displaystyle A = \bigcap_{n \in \N} \left( \bigsqcup_{B \in \cG^{\delta_n}} B \right)$. We obtain $\displaystyle \mu^p (U) = \mu^p (A)$ and $\displaystyle A \subset \bigsqcup_{B \in \cG^{\delta_n}} B$ for any $n$. Thus, owing to \eqref{proof_positive_measure_case_1}, we have
\[
\mu^p_{\delta_n} (A) \leq \mu^p_{\delta_n} \left( \bigsqcup_{B \in \cG^{\delta_n}} B \right) \leq \mu (U) \quad \text{and then} \quad \mu^p (U) = \mu^p(A) \leq \mu (U) \: .
\]
This shows that $\mu^p (U) \leq \mu (U)$, as wanted.

\textbf{Step three.} Since $\mu$ and $\mu^p$ are Borel measures, finite on bounded sets, they are also outer regular (see Theorem~\ref{thm_approximation_Borel_measure}). Then for any Borel set $B \subset X$, and owing to Step two, it holds
\begin{align*}
\mu^p (B)& = \inf \{ \mu^p (U) \: | \: U \textrm{ open, } B \subset U \} \\
         & \leq \inf \{ \mu (U) \: | \: U \textrm{ open, } B \subset U \} = \mu (B) \: .
\end{align*}
Coupling this last inequality with Lemma~\ref{lem_easy_inequality} we obtain $\mu^p = \mu$.
\end{proof}

\subsection{Carath\'eodory's construction for a signed measure}

We recall that a Borel signed measure $\mu$ on $(X,d)$ is an extended real-valued set function $\mu : \cB (X) \rightarrow [-\infty,+\infty]$ such that $\mu (\emptyset)=0$ and, for any sequence of disjoint Borel sets $(A_k)_k$, one has
\begin{equation} \label{eq_countably_additive}
\sum_{k=1}^\infty \mu (A_k) = \mu \left( \bigcup_{k=1}^\infty A_k \right) \: .
\end{equation}

\begin{remk}
Notice that when $\displaystyle \mu \left( \bigcup_{k=1}^\infty A_k \right)$ is finite, its value does not depend on the arrangement of the $A_k$, therefore  the series on the right hand side of \eqref{eq_countably_additive} is commutatively convergent, thus absolutely convergent. In particular, if we write the Hahn decomposition $\mu = \mu^+ - \mu^-$, with $\mu^{+}$ and $\mu^{-}$ being two non-negative and mutually orthogonal measures, then $\mu^+(X)$ and $\mu^-(X)$ cannot be both $+\infty$.
\end{remk}

The question is now the following:
\begin{question}\it
Let $(X,d)$ be a metric space, separable and directionally $(\xi,\eta,\zeta)$--limited at $X$, and let $\mu$ be a Borel signed measure, finite on bounded sets. Is it possible to recover $\mu$ from its values on closed balls by some Carath\'eodory-type construction?
\end{question}

The main difference with the case of a positive measure is that $\mu$ is not monotone and thus the previous construction is not directly applicable. A simple idea could be to rely on the Hahn decomposition of $\mu$: indeed, 
$\mu^+$ and $\mu^-$ are positive Borel measures, and since one of them is finite, both are finite on bounded sets (recall that $\mu$ is finite on bounded sets by assumption). Once again, we cannot directly apply Carath\'eodory's construction to $\mu^+$ or $\mu^-$ since we cannot directly reconstruct $\mu^+ (B)$ and $\mu^- (B)$ simply knowing $\mu(B)$ for any closed ball $B$. We thus try to apply Carath\'eodory's construction not with $\mu^+ (B)$, but with $\left( \mu (B) \right)_+$, where $a_+$ (resp. $a_-$) denote the positive part $\max (a,0)$ (resp. the negative part $\max(-a,0)$) for any $a \in \R$. To be more precise, we state the following definition.

\begin{dfn} \label{caratheodory_construction_signed}
Let $\mu$ be a Borel signed Radon measure in $X$ and let $\cC$ be the family of closed balls in $X$. We define
\[
\begin{array}{lcrclclcrcl}
p_+ & : & \cC  & \longrightarrow & \R_+                                  & \text{and} &  p_- & : & \cC  & \longrightarrow & \R_+\\
    &   & B  &   \longmapsto     & \displaystyle \left( \mu(B) \right)_+ & & &   & B  &   \longmapsto     & \displaystyle \left( \mu(B) \right)_- \: .
\end{array}
\]
Then according to Carath\'eodory's construction, we define the metric outer measure $\mu^{p_+, \ast}$ such that for any $A \subset X$,
\[
\mu^{p_+, \ast}(A) = \lim_{\delta \to 0} \mu^{p_+, \ast}_\delta(A) = \lim_{\delta \to 0} \inf \left\lbrace \sum_{i=0}^\infty p_+ (A_i) \,\bigg| \, A \subset \bigcup_{i\in \N} A_i, \text{ and for all } i, \: A_i\in \cC,\ \diam(A_{i})\le \delta  \right\rbrace \: .
\]
Similarly we define $\mu^{p_-,\ast}$ and then call $\mu^{p_+}$ and $\mu^{p_-}$ the Borel measures associated with $\mu^{p_+,\ast}$ and $\mu^{p_-,\ast}$. Finally, we set $\mu^p = \mu^{p_+} - \mu^{p_-}$.
\end{dfn}

\begin{theo} \label{thm_signed_case}
Let $(X,d)$ be a metric space, separable and directionally $(\xi,\eta,\zeta)$--limited at $X$ and let $\mu = \mu^+ - \mu^-$ be a Borel signed measure on $X$, finite on bounded sets. Let $\mu^{p} = \mu^{p_+}-\mu^{p_-}$ be as in Definition~\ref{caratheodory_construction_signed}. Then $\mu^p = \mu$.
\end{theo}

\begin{proof} 
We observe that $\mu^{p_+}$ and $\mu^{p_-}$ are Borel measures: indeed, by construction they are metric outer measures and Carath\'eodory criterion implies that they are Borel measures. Furthermore, for any closed ball $B \in \cC$, if we set $p(\mu^+) (B) = \mu^+ (B)$ (note that $p(\mu^+)$ is the canonical premeasure associated with $\mu^+$ while $p_+$ is not a priori associated with any measure) then
\[
p_+ (B) = \left( \mu (B) \right)_+ \leq \mu^+ (B) = p(\mu^+) (B) \: ,
\]
thus by construction, $\mu^{p_+,\ast} \leq \mu^{p(\mu^+), \ast}$ and then one gets $\mu^{p_+} \leq \mu^{p(\mu^+)}$; similarly one can show that $\mu^{p_-} \leq \mu^-$. Thanks to Proposition~\ref{prop_positive_case}, we have proven that
\begin{equation} \label{eq_first_ineq_signed_case}
\mu^{p_+} \leq \mu^+ \quad \text{and} \quad \mu^{p_-} \leq \mu^- \: .
\end{equation}
In particular, $\mu^{p_+}$ and $\mu^{p_-}$ are finite on bounded sets, as it happens for $\mu^+$ and $\mu^-$.

Let now $A \subset X$ be a Borel set. It remains to prove that $\mu^{p_+} (A) = \mu^{p_+,\ast} (A) \geq \mu^+ (A)$ (and the same for $\mu^{p_-}$). We argue exactly as in the proof of Lemma~\ref{lem_easy_inequality}. Let $\delta >0$, then for any $\eta > 0$ there exists a countable collection of closed balls $( B_j^\eta )_{j \in \N} \subset \cC_\delta$ such that $A \subset \bigcup_j B_j^\eta$ and $\displaystyle \mu^{p_+,\ast}_\delta (A) \geq \sum_{j=1}^\infty p_+ (B_j^\eta) - \eta$, so that
\[
\mu^{p_+,\ast}_\delta (A) + \eta \geq \sum_{j=1}^\infty p_+ (B_j^\eta) = \sum_{j=1}^\infty \left( \mu (B_j^\eta) \right)_+ \geq \sum_{j=1}^\infty \mu (B_j^\eta) \geq \mu \left( \bigcup_j B_j^\eta  \right) \geq \mu (A) \: .
\]
Letting $\eta \rightarrow 0$ and then $\delta \rightarrow 0$ gives
\begin{equation} \label{proof_signed_case_1}
\mu (A) \leq \mu^{p_+,\ast} (A) = \mu^{p_+} (A) \: .
\end{equation}
We recall that $\mu^+$ and $\mu^-$ are mutually singular, hence there exists a Borel set $P \subset X$ such that for any Borel set $A$
\[
\mu^+ (A) = \mu (P \cap A) \quad \textrm{and} \quad \mu^- (A) = \mu ( (X-P) \cap A ) \: .
\]
Thanks to \eqref{proof_signed_case_1} we already know that $\mu \leq \mu^{p_+}$, therefore we get $\mu^+ (A) = \mu (P \cap A) \leq \mu^{p_+} (P \cap A) \leq \mu^{p_+} (A)$ for any Borel set $A$. Thanks to \eqref{eq_first_ineq_signed_case}, we finally infer that $\mu^{p_+} = \mu^+$, $\mu^{p_-} = \mu^-$, i.e., that $\mu^p = \mu$.
\end{proof}

\begin{remk}
If $\mu$ is a vector-valued measure on $X$, with values in a finite vector space $E$, we can apply the same construction componentwise. \end{remk}

\section{Recovering measures from approximate values on balls}

We now want to reconstruct a measure $\mu$ (or an approximation of $\mu$) starting from approximate values on closed balls, given by a premeasure $q$ satisfying \eqref{eq_mean_premeasure}.
More precisely, we can now reformulate Question 2 in the context of directionally limited metric spaces.

\begin{question}\it
Let $(X,d)$ be a separable metric space, directionally $(\xi,\eta,\zeta)$--limited at $X$ and let $\mu$ be a positive Borel measure on $X$. Is it possible to reconstruct $\mu$ from $q$ up to multiplicative constants? Can the same be done when $\mu$ is a signed measure?
\end{question}

In section \ref{section:2.1} below we explain with a simple example involving a Dirac mass why Carath\'eodory's construction does not allow to recover $\mu$ from the premeasure $q$ defined in \eqref{varif-prem}. Then in section \ref{sect:packtype} we define a \emph{packing construction} of a measure, that is in some sense dual to Carath\'eodory Method II, and we show that in a directionally limited and separable metric space $(X,d)$, endowed with an asymptotically $(\a , \g)$-bounded measure $\nu$ (see \eqref{agbounded}) it produces a measure equivalent to the initial one. 

\subsection{Why Carath\'eodory's construction is not well-suited}
\label{section:2.1}
Let us consider a Dirac mass $\mu = \delta_x$ in $\R^n$ and define
\[
q (B_r (y)) = \frac{1}{r} \int_{s=0}^r \mu (B_s (y)) \, ds\,.
\]
It is easy to check that this particular choice of premeasure $q$ satisfies \eqref{eq_mean_premeasure}. 
First of all, for any $r > 0$,
\[
q (B_r(x)) = \frac{1}{r} \int_{s=0}^r \delta_x (B_s (x)) \, ds = \frac{1}{r} \int_{s=0}^r 1 \, ds = 1 \: .
\]
If now $y$ is at distance $\eta$ from $x$ for some $0<\eta<r$, we have
\[
q (B_r (y)) = \frac{1}{r} \int_{s=0}^r \delta_x (B_s (y)) \, ds = \frac{1}{r} \int_{s=\eta}^r 1 \, ds = \frac{r-\eta}{r} \: .
\]
Therefore, $q(B_{r}(y)) \to 0$ as $d(x,y) = \eta \to r$. We can thus find a covering made by a single ball of radius less than $r$ for which $\mu_r^q ( \{x\} )$ is as small as we wish. This shows that Carath\'eodory's construction associated with this premeasure produces the zero measure.
\begin{center}
\setcounter{subfigure}{0}
\begin{figure}[!htp]
\subfigure[Bad covering of a Dirac mass]{\includegraphics[width=0.25\textwidth]{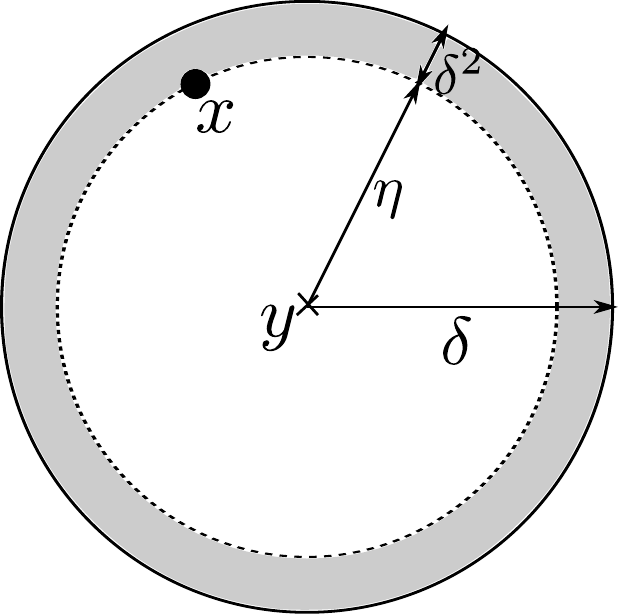}}
\quad
\subfigure[Bad covering of a curve]{\includegraphics[width=0.60\textwidth]{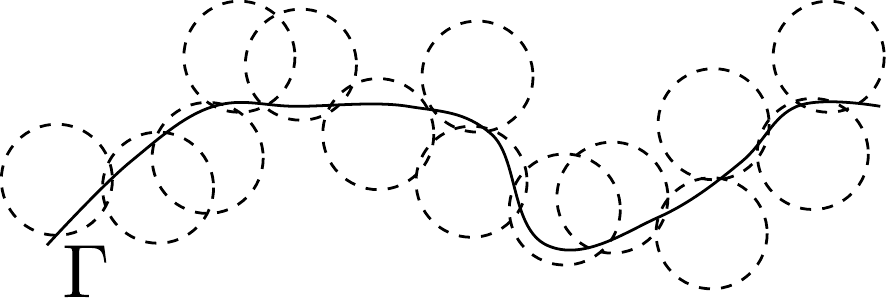}}
\end{figure}
\end{center}
More generally, as soon as it is possible to cover with small balls such that the mass of the measure inside each ball is close to the boundary of the ball, one sees that Carath\'eodory's construction ``looses mass''. For instance, take $\mu = \cH^1_{| \Gamma}$, where $\Gamma\subset \R^{n}$ is a curve of length $L_\Gamma$ and $\cH^{1}$ is the $1$-dimensional Hausdorff measure in $\R^{n}$, then cover $\Gamma$ with a family of closed balls $\cB_\delta$ of radii $\delta$ with centers at distance $\eta$ from $\Gamma$. Assuming that no portion of the curve is covered more than twice, then
\begin{align*}
\sum_{B \in \cB_\delta} q(B)&  = \sum_k \frac{1}{\delta} \int_{s=0}^\delta \mu (B_s (x_k))  = \sum_k \frac{1}{\delta} \int_{s=\eta}^\delta \mu (B_s (x_k)) \, ds \\
& \leq \frac{\delta - \eta}{\delta} \sum_k \mu (B_\delta (x_k)) \\
& \leq 2 L_\Gamma \frac{\delta - \eta}{\delta} \xrightarrow[\delta \to 0]{} 0 \: ,
\end{align*}
with $\eta = \delta - \delta^2$ for instance.

The same phenomenon cannot be excluded by blindly centering balls on the support of the measure $\mu$. Indeed, take a line $\ell\subset \R^{2}$ with a Dirac mass placed on it at some $x\in \ell$, so that $\mu = \cH^1_{| \ell} + \delta_x$. Then, by centering the balls on the support of $\mu$, we may recover the Hausdorff measure restricted to $\ell$, but not the Dirac mass, for the same reason as before. 

We thus understand that the position of the balls should be optimized in order to avoid the problem. For this reason we consider an alternative method, based on a packing-type construction.

\begin{figure}[!htp]
\begin{center}
\includegraphics[width=0.50\textwidth]{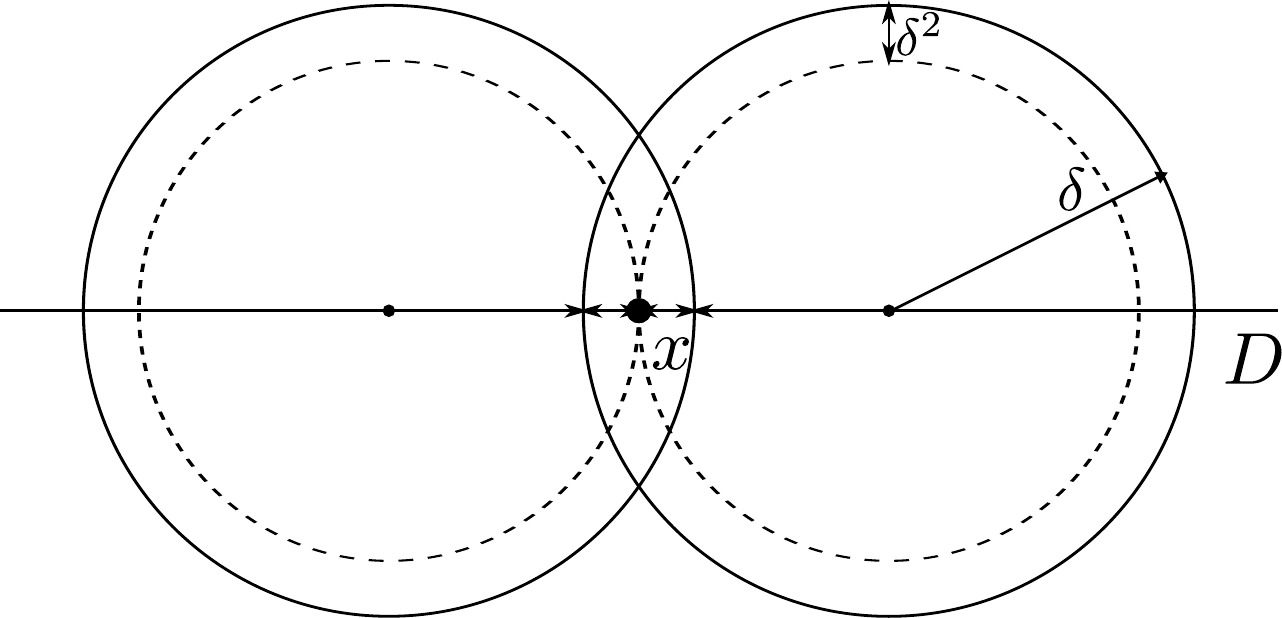}
\caption{Bad covering with balls centered on the support of the measure}
\end{center}
\end{figure}

\subsection{A packing-type construction}
\label{sect:packtype}

Taking into account the problems described in the examples of the previous section, one realizes the need to optimize the position of the centers of the balls in order to properly reconstruct the original measure $\mu$. The idea is to consider a kind of dual construction, that is, a supremum over packings rather than an infimum over coverings. To this aim we recall a notion of  \emph{packing} of balls.

\begin{dfn}[Packings] \label{def:packing}
Let $(X,d)$ be a separable metric space and $U \subset X$ be an open set. We say that $\cF$ is a packing of $U$ of order $\delta$ if $\cF$ is a countable family of disjoint closed balls whose radius is less than $\delta$ and such that
\[
\bigsqcup_{B \in \cF} B \subset U \: .
\]
\end{dfn}

\begin{dfn}[Packing construction of measures]  \label{dfn_packing_construction}
Let $(X,d)$ be a separable metric space and let $q$ be a premeasure defined on closed balls. Let $U \subset X$ be an open set and fix $\delta >0$. We set
\[
\hmu_\delta^{q} (U) := \sup \left\lbrace \sum_{B \in \cF} q (B) \, : \, \cF \textrm{ is a packing of order } \delta \textrm{ of } U \right\rbrace 
\]
and, in a similar way as in Carath\'eodory construction, we define
\[
\hmu^{q} (U) = \lim_{\delta \to 0} \hmu_{\delta}^{q} (U) = \inf_{\delta > 0} \hmu_{\delta}^{q} (U) 
\]
and note that $\delta^\prime \leq \delta$ implies $\hmu_{\delta^\prime}^{q} (U) \leq \hmu_\delta^{q} (U)$. Then, $\hmu^{q}$ can be extended to all $A \subset X$ by setting
\begin{equation} \label{eq:extensionHmuInf}
\hmu^{q}(A) = \inf \left\lbrace \hmu^{q} (U) \, : \, A \subset U, \, U \textrm{ open set} \right\rbrace \, .
\end{equation}
\end{dfn}
The main difference between Definition \ref{dfn_packing_construction} and Carath\'eodory's construction is that the set function $\hmu^{q}$ is not automatically an outer measure: it is monotone but not sub-additive in general.
In order to fix this problem we may apply the construction of outer measures, known as Munroe Method I, to the set function $\hmu^{q}$ restricted to the class of open sets. This amounts to setting, for any $A \subset X$,
\begin{equation}\label{mutilde}
\tilde{\mu}^{q} (A) = \inf \left\lbrace \sum_{n \in \N} \hmu^{q} (U_n) \, : \, A \subset \bigcup_{n \in \N} U_n, \, U_n \textrm{ open set } \right\rbrace \: .
\end{equation}
One can check that $\tilde{\mu}^{q}$ is an outer measure. 

We will prove in Theorem~\ref{thm_main} that, for the class of set functions $q$ we are focusing on, and under additional assumptions on$ X$, $\hmu^{q}$ is already a Borel outer measure. 

\begin{remk} \label{remk_sub_additivity_and_tilde_mu}
Knowing that $\hmu^{q}$ is sub-additive in the class of open sets is enough to show that $\hmu^{q} = \tilde{\mu}^{q}$.
Indeed, the inequality $\tilde{\mu}^{q} (A) \leq \hmu^{q} (A)$ comes directly from the fact that minimizing $\hmu^{q} (U)$ over $U$ open such that $A \subset U$ is a special case of minimizing $\sum_k \hmu^{q} (U_k)$ among countable families of open sets $U_k$ such that $A \subset \bigcup_k U_k$. Assuming in addition that $\hmu^{q}$ is sub-additive on open sets implies that for any countable family of open sets $(U_k)_k$ such that $A \subset \bigcup_k U_k$,
\[
\hmu^{q} (A) \leq \hmu^{q} (\bigcup_k U_k) \leq \sum_k \hmu^{q} (U_k)  \: .
\] 
By definition of $\tilde{\mu}^{q}$, taking the infimum over such families leads to $\hmu^{q} (A) \leq \tilde{\mu}^{q} (A)$.
\end{remk}

\begin{remk}
Our construction of the measure $\hmu^{q}$ is slightly different from the more classical packing construction proposed by Tricot and Taylor in \cite{tricot}. In particular, \eqref{eq:extensionHmuInf} enforces outer regularity of $\hmu^{q}$, while for instance the classical $s$--dimensional packing measure in $\R^n$ associated with the premeasure $q(B_r(x)) = (2r)^s$ is not outer regular for $s < n$ (see \cite{edgar}). 
A more specific comparison between our definition and the one by Tricot and Taylor will be carried out in section \ref{sect:BLversusTT}, where it will be proved that, under the assumption that $q(B) \le C \mu(B)$ for every ball $B\subset \R^{n}$ and for some constant $C>0$ and some Radon measure $\mu$, the two constructions actually produce the same measure.
\end{remk}

\begin{theo} \label{thm_main}
Let $(X,d)$ be a metric space satisfying Hypothesis \ref{hypo1}, let $\mu$ be a positive Borel measure on $X$ and let $q$ be a premeasure given on closed balls satisfying \eqref{eq_mean_premeasure}. Let $\hmu^{q}$ be as in Definition~\ref{dfn_packing_construction}. Then, the following hold:
\begin{enumerate}
\item for any Borel set $A \subset X$,
\begin{equation} \label{eq:generalControl}
\frac{1}{\g C} \mu (A) \leq  \hmu^{q} (A) \leq  C \inf \{ \mu (U) \: | \: A \subset U \text{ open set } \} \: ;
\end{equation}
where $\g$ and $C$ are the constants respectively appearing in Hypothesis \ref{hypo1} and \eqref{eq_mean_premeasure}.
\item $\hmu^q$ is countably sub-additive and is a metric outer measure;
\item if moreover $\mu$ is outer regular, then $\mu$ and the positive Borel measure associated with the outer measure $\hmu^{q}$ (still denoted as $\hmu^{q}$) are equivalent, and more precisely $\displaystyle \frac{1}{\g C} \mu \leq \hmu^q \leq C \mu$.
\end{enumerate}
\end{theo}

We briefly sketch the proof of this result for the reader's convenience.
\begin{enumerate}
\item The first step (see Proposition \ref{prop_sub_additivity_open_case}) is to prove that $\hmu^q$ is countably sub-additive on any sequence of open sets $(U_n)_n$ such that $\sum_n \mu (U_n) < +\infty$. This property does not require Hypothesis $1$ and only relies on the upper bound \eqref{eq_mean_premeasure}--(ii) on $q$.
\item Then, assuming moreover that $\mu$ is finite on bounded sets, we prove in Proposition \ref{prop_sub_additivity_loc_finite_case} the countable sub-additivity of $\hmu^q$.
\item Now we show in Proposition \ref{prop_equivalence_on_open_sets} a crucial fact, i.e. that the finiteness of the measure $\mu$ can be dropped, thus we only require that $\mu$ is a positive Borel measure (not necessarily finite on bounded sets) and that $(X,d)$ satisfies Hypothesis~\ref{hypo1}. The heart of the proof is to use the lower bound \eqref{eq_mean_premeasure} (i) on $q$ and show that it can be transferred to $\hmu^q$, which gives \eqref{eq:generalControl}.
\item Up to this last step, we still do not know that $\hmu^q$ is countably sub-additive. This is proved in Proposition~\ref{cor_countable_sub_additvity}  and follows from the partial sub-additivity result of the first step and the lower bound in \eqref{eq:generalControl}. We stress that this last result does not require that $\mu$ is finite on bounded sets.
\end{enumerate}


\begin{prop} \label{prop_sub_additivity_open_case}
Let $(X,d)$ be a separable metric space and let $\mu$ be a positive Borel measure on $X$. Let $q$ be a premeasure defined on the class $\mathcal C$ of closed balls contained in $X$, such that \eqref{eq_mean_premeasure} (ii) holds. Then, for any countable family of open sets $(U_k)_k \subset X$ satisfying $\sum_{k\in \N} \mu (U_k) < +\infty$, one has
\begin{equation}\label{subadd1}
\hmu^{q} \left( \bigcup_{k \in \N} U_k \right) \leq \sum_{k \in \N} \hmu^{q}  (U_k)\,.
\end{equation}
In particular, if $\mu$ is finite, then $\hmu^{q}$ is an outer measure. 
\end{prop}

\begin{proof} 
Let $( U_k )_k$ be a sequence of open subsets of $X$ such that $\sum_k \mu (U_k) < +\infty$. Let $\epsilon > 0$, then for all $k\in \N$ we define
\[
U_k^\epsilon = \left\lbrace x \in U_k \, : \, d(x , X - U_k) > \epsilon \right\rbrace \: .
\]
Fix $0 < \delta < \frac{\epsilon}{2}$. If $B$ is a closed ball such that $\diam B \leq 2 \delta$ and $B \subset \bigcup_k U_k^\epsilon$, then there exists $k_0$ such that $B \subset U_{k_0}$. Indeed, $B=B_\delta (x)$ and there exists $k_0$ such that $x \in U_{k_0}^\epsilon$ and thus
\[
B_\delta (x) \subset U_{k_0}^{\epsilon - \delta} \subset U_{k_0}^{\frac{\epsilon}{2}} \subset U_{k_0}\: .
\]
Of course the inclusion $B\subset U_{k_{0}}$ remains true for any closed ball $B$ with $\diam B \leq 2 \delta$.
\begin{figure}[!htp]
\begin{center}
\includegraphics[width=0.50\textwidth]{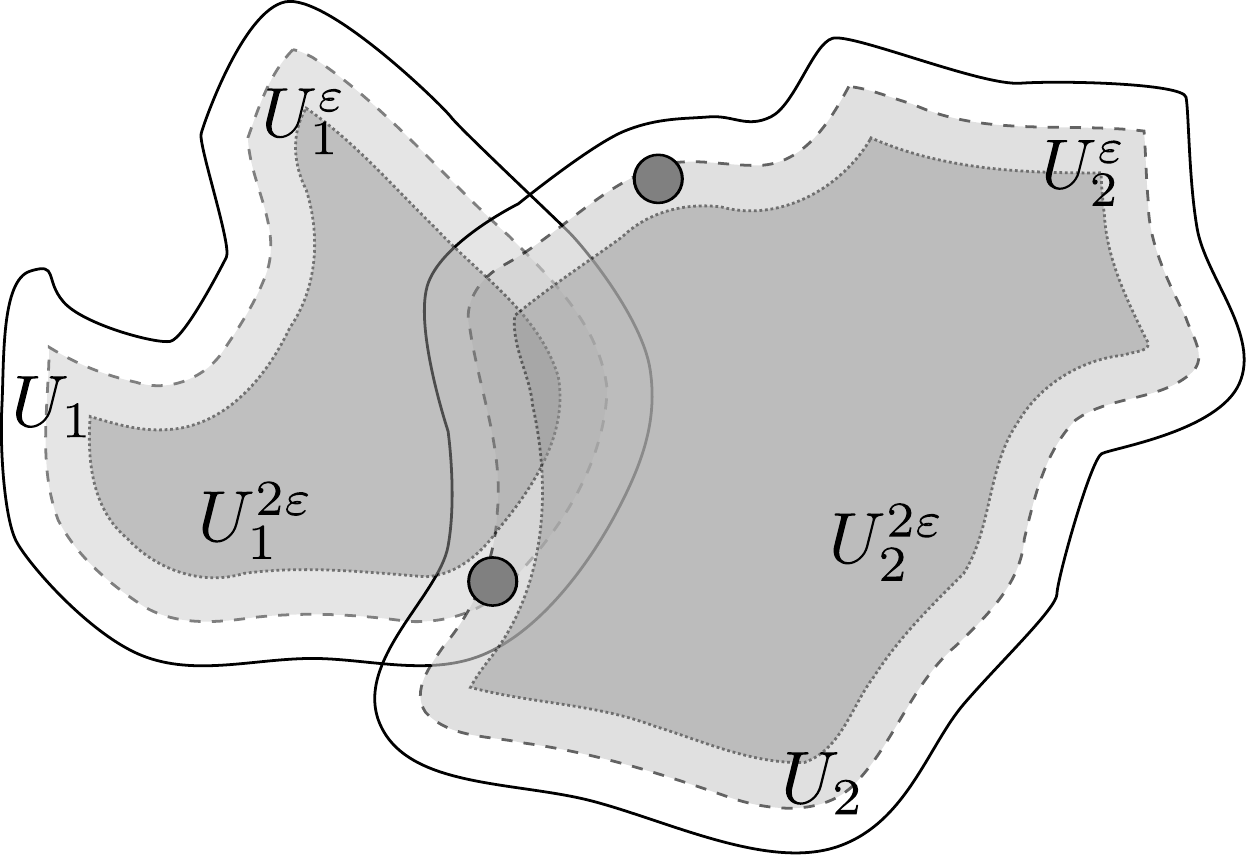}
\caption{Sub-additivity for packing construction}
\end{center}
\end{figure}
Therefore any packing $\cB$ of $\displaystyle \bigcup_k U_k^\epsilon$ of order $\delta$ can be decomposed as the union of a countable family of packings $\cB = \bigsqcup_k \cB_k $, where $\cB_k$ is a packing of $U_k$ of order $\delta$, whence
\[
\sum_{B \in \cB} q(B) = \sum_k \sum_{B \in \cB_k} q(B)\,.
\]
By taking the supremum over all such packings $\cB$ of $\displaystyle \bigcup_k U_k^\epsilon$, we get
\[
\hmu^q_\delta \left( \bigcup_k U_k^\epsilon \right) \leq \sum_k \hmu_\delta^q (U_k) \: .
\]
Then, taking the infimum over $\delta > 0$ and then the supremum over $\epsilon > 0$ gives
\begin{equation} \label{eq_sub_additivity_3}
\sup_{\epsilon > 0} \hmu^{q} \left( \bigcup_k U_k^\epsilon \right)  \leq \inf_{\delta > 0} \sum_{k \in \N} \hmu_\delta^q (U_k) \: .
\end{equation}
We now want to prove that
\begin{equation}\label{eqsupeps}
\sup_{\epsilon > 0} \hmu^{q} \left( \bigcup_{k \in \N} U_k^\epsilon \right) =  \hmu^{q} \left( \bigcup_{k \in \N} U_k \right)\,.
\end{equation}
Let $\cB$ be a packing of $\displaystyle \bigcup_k U_k$ of order $\delta < \frac{\epsilon}{2}$. We have
\begin{equation} \label{eq_sub_additivity_1}
\sum_{B \in \cB} q(B) = \sum_{\substack{B \in \cB \\  B \subset \bigcup_k U_k^\epsilon }} q(B) \: + \sum_{\substack{B \in \cB \\ B \not\subset \bigcup_k U_k^\epsilon }} q(B) \: .
\end{equation}
Notice that since $2 \delta < \epsilon$, for any $B \in \cB$, if $\displaystyle B \not\subset \bigcup_k U_k^\epsilon$ then $\displaystyle B \subset \bigcup_k U_k \setminus \bigcup_k U_k^{2 \epsilon}$. Since $q(B) \le C\mu (B)$ according to \eqref{eq_mean_premeasure} (ii), we get
\begin{equation} \label{eq_sub_additivity_2}
\sum_{\substack{B \in \cB \\ B \not\subset \bigcup_k U_k^\epsilon }} q(B) \leq C\sum_{\substack{B \in \cB \\ B \not\subset \bigcup_k U_k^\epsilon }} \mu(B) = C\mu \left( \bigsqcup_{\substack{B \in \cB \\ B \not\subset \bigcup_k U_k^\epsilon }} B \right) \leq C\mu \left( \bigcup_k U_k \setminus \bigcup_k U_k^{2 \epsilon} \right)\, .
\end{equation}
Owing to the fact that $\displaystyle \bigcup_k U_k = \bigcup_{\substack{\text{countable} \\ \epsilon > 0}} \bigcup_k U_k^{2 \epsilon}$, and $\bigcup_k U_k^{2 \epsilon}$ is decreasing in $\epsilon$, we have that
\begin{equation}\label{muvazero}
\mu \left( \bigcup_k U_k \setminus \bigcup_k U_k^{2 \epsilon} \right) \xrightarrow[\epsilon \to 0]{} 0 
\end{equation}
as soon as $\mu (\bigcup_k U_k) < +\infty$, which is true under the assumption $ \sum_k \mu (U_k) < +\infty$. Therefore, by \eqref{eq_sub_additivity_1}, \eqref{eq_sub_additivity_2} and \eqref{muvazero} we infer that
\begin{equation}\label{eqsubadd3}
\sum_{B \in \cB} q(B) \leq \sum_{\substack{B \in \cB \\  B \subset \bigcup_k U_k^\epsilon }} q(B)  + C \mu \left( \bigcup_k U_k \setminus \bigcup_k U_k^{2 \epsilon} \right) \: .
\end{equation}
Taking the supremum in \eqref{eqsubadd3} over all packings $\cB$ of order $\delta$ of $\bigcup_k U_k$, we get
\begin{align*}
\hmu_\delta^q \left( \bigcup_k U_k \right) & \leq \sup \left\lbrace \sum_{\substack{B \in \cB \\  B \subset \bigcup_k U_k^\epsilon }} q(B) \: : \: \cB \text{ is a packing of } \bigcup_k U_k \text{ order } \delta \right\rbrace  + C\mu \left( \bigcup_k U_k \setminus \bigcup_k U_k^{2 \epsilon} \right) \\
& \leq \hmu_\delta^q \left( \bigcup_k U_k^\epsilon \right) + C\mu \left( \bigcup_k U_k \setminus \bigcup_k U_k^{2 \epsilon} \right) \: .
\end{align*}
Then taking the limit as $\delta\to 0$ we obtain
\[
\hmu^{q} \left( \bigcup_k U_k \right) \leq \hmu^{q} \left( \bigcup_k U_k^\epsilon \right) + C\mu \left( \bigcup_k U_k \setminus \bigcup_k U_k^{2 \epsilon} \right) \: 
\]
and finally, letting $\epsilon\to 0$, we prove \eqref{eqsupeps}.

We now turn to the right hand side of \eqref{eq_sub_additivity_3}. For fixed $k$, $\hmu^q_\delta (U_k)$ is decreasing when $\delta \downarrow 0$, therefore
\begin{equation} \label{eq_sub_additivity_4}
\lim_{\delta \downarrow 0} \sum_k \hmu_\delta^q (U_k) = \sum_k \lim_{\delta \downarrow 0} \hmu_\delta^q (U_k) = \sum_k \hmu^q (U_k)
\end{equation}
provided that $ \sum_k \hmu_\delta^q (U_k)$ is finite for some $\delta > 0$, which is true since $q(B) \leq C\mu(B)$, $\hmu_\delta^q (U_k) \leq C\mu (U_k)$ for all $k$, so that $\sum_k \hmu_\delta^q (U_k) \leq C\sum_k \mu(U_k) < +\infty$. 
Finally, thanks to \eqref{eq_sub_additivity_3}, \eqref{eqsupeps} and \eqref{eq_sub_additivity_4} we obtain the conclusion.
\end{proof}

\begin{prop} \label{prop_sub_additivity_loc_finite_case}
Let $(X,d)$ be a separable metric space and let $\mu$ be a positive Borel measure on $X$, finite on bounded sets. Let $q$ be a premeasure defined on the class $\mathcal C$ of closed balls contained in $X$, such that \eqref{eq_mean_premeasure} (ii) holds. Then $\hmu^q$ is countably sub-additive, thus it is an outer measure.
\end{prop}

\begin{proof}
Let $(A_k)_k$ be a countable family of disjoint sets such that $\mu \left( \bigsqcup_k A_k \right) < +\infty$. We shall prove that
\begin{equation} \label{eq_sub_additivity_6}
\hmu^{q} \left( \bigsqcup_k A_k \right) \leq \sum_k \hmu^{q} (A_k) \: .
\end{equation}
Let $\epsilon >0$. By outer regularity of $\mu$ (since $\mu$ a Borel measure finite on bounded sets, it is outer regular by Theorem~\ref{thm_approximation_Borel_measure}) and by definition of $\hmu^q$, let $\left( U_k \right)_k$ be a family of open sets such that, for any $k$,
\[
A_k \subset U_k \quad \textrm{and} \quad \mu (U_k ) \leq \mu (A_k) + \frac{1}{2^k} \quad \textrm{and} \quad \hmu^q (U_k ) \leq \hmu^q (A_k) + \frac{\epsilon}{2^k} \: .
\]
Hence $\sum \mu (U_k) \leq \sum \mu(A_k) + 1 < +\infty$ and by Proposition~\ref{prop_sub_additivity_open_case} we thus find
\[
\hmu^{q} \left( \bigsqcup_k A_k \right) \leq \hmu^{q} \left( \bigcup_k U_k \right) \leq \sum_k \hmu^{q} (U_k) \leq \sum_k \hmu^{q} (A_k) + \epsilon \: .
\]
Let $\epsilon \rightarrow 0$ to get \eqref{eq_sub_additivity_6}.

The case of a countable family $(A_k)_k$ such that $\mu \left( \bigcup_k A_k \right) < +\infty$ is obtained from the case of disjoint sets in the standard way, by defining $B_k \subset A_k$ as
\[
B_k = A_k - \bigcup_{i=1}^{k-1} A_i \: .
\]
The family $(B_k)_k$ is disjoint and $\displaystyle \bigsqcup_{k \in \N} B_k = \bigcup_{k \in \N} A_k$, hence by \eqref{eq_sub_additivity_6} one gets
\[
\hmu^{q} \left( \bigcup_k A_k \right) = \hmu^{q} \left( \bigsqcup_k B_k \right) \leq \sum_k \hmu^{q} (B_k) \leq \sum_k \hmu^{q} (A_k) \: .
\]
Finally, let us consider the general case of $(A_k)_k$ being any countable family of sets. Let $(X_n)_n$ be an increasing family of bounded sets such that $\cup_n X_n = X$, then for all $n$, $\cup_k ( A_k \cap X_n )$ is bounded, hence $\mu \left( \cup_k ( A_k \cap X_n ) \right) < +\infty$ and therefore
\[
\hmu^q \left( \bigcup_k ( A_k \cap X_n ) \right) \leq \sum_k \hmu^q \left( A_k \cap X_n \right) \: . 
\]
We let $n \rightarrow + \infty$ and we infer by monotone convergence that
\[
\hmu^q \left( \bigcup_k A_k \right) \leq \sum_k \hmu^q \left( A_k  \right) \: .
\]
\end{proof}


In order to have the countable sub-additivity of $\hmu^{q}$ (in the case where $\mu$ is not assumed to be finite on bounded sets) it is enough to show that $\sum_k \mu(U_k) = +\infty$ implies $\sum_k \hmu^{q} (U_k) = +\infty$. If this is true, then either $\sum_k \mu(U_k) < +\infty$ and the sub-additivity is given by Proposition~\ref{prop_sub_additivity_open_case}, or $\sum_k \hmu^{q} (U_k) = +\infty$ and the sub-additivity is immediate. Now the core of the problem is to estimate $\hmu^{q}$ from below by $\mu$. The main issue is that the lower bound \eqref{eq_mean_premeasure} (i) does not imply the stronger lower bound 
\[
q(B)\geq C^{-1} \mu (B)\,.
\]
Moreover, unless we know that the measure $\mu$ is doubling, there is no reason to expect that the inequality $\mu (2B) \le c\mu (B)$  is verified for some $c\ge 1$  (where by $2B$ we denote the ball concentric to $B$ with double radius) and for any ball $B$. Nevertheless, by assuming that $(X,d)$ is directionally bounded and by comparing $\mu$ with an asymptotically doubling measure $\nu$ on $X$, which we assume to exist, we are able to prove that a doubling property for $\mu$ actually holds for enough balls, so that we can choose packings among these balls. Before showing the result, we need to introduce the notion of \textit{asymptotically $(\a,\g)$-bounded measure}. We say that $\nu$ is an asymptotically $(\a,\g)$-bounded measure on $(X,d)$ if it is finite on bounded sets and strictly positive on any ball with positive radius, and for all $x\in X$ it satisfies
\begin{equation}\label{agbounded}
\limsup_{r\to 0^{+}} \frac{\nu(B_{r}(x))}{\nu(B_{\a r}(x))} \le \g \,.
\end{equation}

\begin{remk}
Notice that if we assume that $\nu$ is asymptotically doubling on $X$, that is,  $\nu$ is finite on bounded sets and there exists a constant $d \ge 1$ such that for all $x\in X$ it holds
\begin{equation*}
\limsup_{r\to 0^{+}} \frac{\nu(B_{2r}(x))}{\nu(B_{r}(x))} \le d \,,
\end{equation*}
then for any $\a\in (0,1]$, taking $Q$ as the unique integer such that $2^{-(Q+1)}< \a\le 2^{-Q}$, one can easily check that $\nu$ is asymptotically $(\a,d^{Q+1})$-bounded.
\end{remk}

We conveniently state some key properties on the metric space $(X,d)$, that will be constantly assumed in the rest of this section. 
\begin{hypothesis}\label{hypo1}
$(X,d)$ is a directionally limited metric space endowed with an asymptotically $(\a , \g)$-bounded measure $\nu$ satisfying \eqref{agbounded} for some constants $\alpha \in (0,1]$ and $\gamma \ge 1$.
\end{hypothesis}

\begin{prop} \label{prop_diffusion_in_X}
Assume that $(X,d,\nu)$ satisfy Hypothesis \eqref{hypo1} and let $\mu$ be a positive Borel measure on $X$. Let 
\[
A_0 = \left\lbrace x \in X \: : \: \liminf_{r \to 0^{+}} \frac{\mu (B_{r}(x))}{\nu (B_r(x))} = 0 \right\rbrace \quad \textrm{and} \quad A_+ = \left\lbrace x \in X \: : \: 0 < \liminf_{r \to 0^{+}} \frac{\mu (B_r(x))}{\nu (B_r(x))} \leq +\infty \right\rbrace \: .
\]
Then the following hold.
\begin{enumerate}[(i)]
\item For all $x \in A_+$, either $\mu (B_r(x)) = +\infty$ for all $r>0$, or
\[
\liminf_{r \to 0^{+}} \frac{\mu (B_{ r} (x))}{\mu (B_{\a r}(x))} \leq \gamma \: .
\]

\item $\mu (A_0) = 0$.
\end{enumerate}
In particular, for a fixed $\epsilon_0 >0$ and for $\mu$--almost any $x \in X$, there exists a decreasing infinitesimal sequence $(r_{n})_{n}$ of radii (depending on $\epsilon_0$ and $x$), such that
\begin{equation}\label{epsliminf}
\mu (B_{r_{n}} (x)) \le (\gamma+\epsilon_0) \, \mu(B_{\a r_{n}}(x))\,,\quad \forall\, n\in \N  \: .
\end{equation}
\end{prop}


\begin{proof}

\textbf{Proof of (i).} Let $x \in A_+$. By monotonicity, either $\mu (B_r(x)) = +\infty$ for all $r>0$ (and then \eqref{epsliminf} is also trivially satisfied) or there exists some $R$ such that, for all $r \leq R$, $\mu (B_r (x)) < +\infty$. In this case the function defined by 
\[
\displaystyle f(r) = \frac{\mu(B_r(x))}{\nu(B_r(x))}
\]
is non-negative and finite for $r$ small enough. Moreover, since $x \in A_+$ we have $\liminf_{r \to 0} f(r) > 0$. Let us prove that
\begin{equation}\label{limsupf}
\liminf_{r \to 0} \frac{f( r)}{f(\a r)} \leq 1 \: .
\end{equation}
Assume by contradiction that $\displaystyle \liminf_{r \to 0} \frac{f( r)}{f(\a r)} > 1$, then necessarily $\a < 1$ and there exists $r_0 > 0$ and $\beta > 1$ such that for all $r \leq r_0$, $f( r) \geq \beta f(\a r)$. Consider now the sequence $(r_k)_k$ defined by $r_k = \a^{k} r_0$. Then $r_k \rightarrow 0$ and
\[
f(r_k) \leq \beta^{-1} f(\a^{-1} r_k) = \beta^{-1} f(r_{k-1}) \leq \beta^{-k} f(r_0) \xrightarrow[k \to \infty]{} 0
\]
which contradicts $\liminf_{r \to 0} f(r) > 0$ and thus proves \eqref{limsupf}. Let us notice that $\mu(B_{r}(x))>0$ for all $r>0$ since $x\in A_{+}$ and that, by definition, we have
\[
\frac{\mu(B_{ r}(x))}{\mu(B_{\a r}(x))} = \frac{f( r)}{f(\a r)} \cdot \frac{\nu(B_{ r}(x))}{\nu(B_{\a r}(x))}\,.
\]
Since $\nu$ is asymptotically $(\a, \g)$-bounded, by \eqref{agbounded} we get
\[
\liminf_{r \to 0^{+}} \frac{\mu(B_{ r}(x))}{\mu(B_{\a r}(x))} \leq  \limsup_{r \to 0^{+}} \frac{\nu(B_{ r}(x))}{\nu(B_{\a r}(x))} \liminf_{r \to 0^+} \frac{f(\a r)}{f(r)} \leq \gamma \: .
\]

\textbf{Proof of (ii).} Let us show that $\mu (A_0) = 0$. First assume that $\nu (X) < +\infty$ and let $\epsilon > 0$. Consider
\[
\cF_\epsilon = \left\lbrace B \subset X \: | \: B=B_r(a), a \in A_0 \text{ and } \mu(B) \leq \epsilon \nu (B) \right\rbrace \: .
\]
Let $a \in A_0$ be fixed. Since $\displaystyle \liminf_{r \to 0^{+}} \frac{\mu(B_r(a))}{\nu(B_r(a))} = 0$, there exists $r >0$ such that $B_r (a) \in \cF_\epsilon$. Every point in $A_0$ is the center of some ball in $\cF_\epsilon$, so that we can apply Theorem \ref{thm_generalized_besicovic_federer_covering_lemma} and obtain $2\zeta +1$ countable families $\cG_1, \ldots \cG_{2\zeta +1} $ of disjoint balls in $\cF_\epsilon$, such that
\[
A_0 \subset \bigcup_{j=1}^{2\zeta +1} \bigsqcup_{B \in \cG_j} B \: .
\]
Therefore
\[
\mu (A_0) \leq \sum_{j=1}^{2\zeta +1} \sum_{B \in \cG_j} \underbrace{\mu (B)}_{\leq \epsilon \nu (B)} \leq \epsilon \sum_{j=1}^{2\zeta +1} \nu \left( \bigsqcup_{B \in \cG_j} B \right) \leq \epsilon (2\zeta +1) \nu (X) \: .
\]
Hence $\mu (A_0) = 0$ if $\nu (X) < +\infty$. Otherwise, replace $X$ by $X \cap U_k (x_{0})$ for any fixed $x_{0}\in X$ to obtain that for any $k \in \N$, $\mu (A_0 \cap U_k (x_{0})) = 0$, then let $k\to\infty$ to conclude that $\mu (A_0) = 0$.

Finally \eqref{epsliminf} is an immediate consequence of the fact that $X = A_{0}\cup A_{+}$ coupled with (i) and (ii).
\end{proof}


\begin{corollary}[Besicovitch with doubling balls] \label{bescivoic_with_doubling_balls}
Assume that $(X,d,\nu)$ satisfy Hypothesis \eqref{hypo1} and let $\mu$ be a positive Borel measure on $X$. Fix $\epsilon_{0}>0$ and for any $\delta > 0$ define
\[
\cF_\delta = \left\lbrace B=B_{r}(x) \text{ closed ball } \subset X \: : \: \mu(B) \leq (\gamma+\epsilon_0) \mu (B_{\a r}(x)) \text{ and } \diam B \leq 2\delta \right\rbrace \: 
\]
and, for any $A \subset X$, 
\[
\cF_\delta^A = \{ B \in \cF_\delta \: : \: B=B_r (a),\ a \in A \}\: .
\] 
Then there exist $A_0 \subset X$ and $2\zeta +1$ countable subfamilies of $\cF_\delta^A$ of disjoint closed balls, $\cG_1, \ldots \cG_{2\zeta +1}$ such that
\[
A \subset A_0 \cup \bigcup_{j=1}^{2\zeta +1} \bigsqcup_{B \in \cG_j} B\quad \text{ and }\quad \mu (A_0) = 0 \: .
\]
Moreover, if $\mu (A) < +\infty$, then for any open set $U \subset X$ there exists a countable collection $\cG$ of disjoint balls in $\cF_\delta^A$ such that
\[
\bigsqcup_{B \in \cG} B \subset U \quad \text{and} \quad \mu \left( (A\cap U) \setminus \bigsqcup_{B \in \cG} B \right) = 0 \: .
\]
\end{corollary}

\begin{proof}
Thanks to \eqref{epsliminf} (see Proposition~\ref{prop_diffusion_in_X}) we know that for $\mu$-almost every $x \in X$ there exists a decreasing infinitesimal sequence $(r_{n})_{n}$ such that $\displaystyle B_{r_{n}}(x) \in \cF_\delta$ for all $n\in \N$. Hence for $\mu$-almost any $x \in X$ we have
\[
\inf \left\lbrace r \: | \: B_r (x) \in \cF_\delta \right\rbrace = 0 \: .
\]
Then, the conclusion follows from Theorems \ref{thm_generalized_besicovic_federer_covering_lemma} and \ref{thm_generalized_besicovic_federer}.
\end{proof}

We can now prove that $\hmu^{q}$ and $\mu$ are equivalent on Borel sets (as set functions since we have not completely proved the sub-additivity of $\hmu^{q}$ yet).

\begin{prop} \label{prop_equivalence_on_open_sets}
Let $(X,d,\nu)$ satisfy Hypothesis \eqref{hypo1} with constants $(\a,\g)$ and let $\mu$ be a positive Borel measure on $X$. Let $q$ be a premeasure satisfying \eqref{eq_mean_premeasure} with constants $\a$ and $C$. Let $\hmu^{q}$ be as in Definition~\ref{dfn_packing_construction}. 
Then for any Borel set $A \subset X$ we have 
\[
\frac{1}{\gamma C} \mu (A) \leq  \hmu^{q} (A) \leq C \inf \{ \mu (U) \: | \: A \subset U \text{ open set } \} \, .
\]
Therefore if $\mu$ is outer regular then
\[
\frac{1}{\g C} \mu (A) \leq  \hmu^{q} (A) \leq  C\mu (A) \: .
\]
\end{prop}

\begin{proof}
Let $U \subset X$ be an open set, then the inequality $\hmu^{q} (U) \leq C\mu (U)$ is just a consequence of the definition of $\hmu^{q}$ and of the second inequality in \eqref{eq_mean_premeasure}, i.e., of the fact that, for any closed ball $B$, $q(B) \leq C\mu(B)$. Now we prove the other inequality by splitting the problem into two cases. 

\textit{Case $\mu(U) < +\infty$.} Let $\epsilon_0 >0$ and $\delta > 0$, then we can apply Corollary~\ref{bescivoic_with_doubling_balls} (Besicovitch with doubling balls) to get a countable family $\cG_\delta$ of disjoint balls of 
\[
\cF_\delta^U =  \left\lbrace B=B_r(x) \subset U \: : \: \mu( B) \leq (\g + \epsilon_0) \mu (B_{\a r}(x)) \text{ and } \diam B \leq 2\delta \right\rbrace
\]
such that
\[
\mu(U)  = \mu \left( \bigsqcup_{B \in \cG_\delta} B \right) \quad \textrm{and} \quad  \bigsqcup_{B \in \cG_\delta} B \subset U \: .
\]
Therefore by \eqref{eq_mean_premeasure} and by definition of $\hmu_\delta^q $,
\begin{align*}
\hmu_\delta^q (U) \geq \sum_{B \in \cG_\delta} q(B) \geq &\sum_j \frac{1}{C} \mu ( B_{\a r_j} (x_j) ) \\
& \geq \frac{1}{C(\gamma+\epsilon_0)} \sum_j \mu (B_{r_j} (x_j)) = \frac{1}{C(\gamma+\epsilon_0)} \mu \left( \bigsqcup_{B \in \cG_\delta} B \right) = \frac{1}{C(\gamma+\epsilon_0)} \mu (U) \: .
\end{align*}
Letting $\delta\to 0$ and then $\epsilon_0 \rightarrow 0$ gives $\displaystyle \hmu^{q}(U) \geq \frac{1}{ \gamma C} \mu(U)$.

\textit{Case $\mu(U) = +\infty$.} Let $\delta > 0$ and $\cF_{\delta}^{U}$ be as in the previous case, then applying Corollary~\ref{bescivoic_with_doubling_balls} (Besicovitch with doubling balls) 
gives $2\zeta +1$ countable families $\cG_\delta^1, \ldots , \cG_\delta^{2\zeta +1}$ of balls in $\cF_\delta^{U}$ such that
\[
U  \subset U_0 \cup \bigcup_{j=1}^{2\zeta +1} \bigsqcup_{B \in \cG_\delta^j} B\quad \text{with}\quad \mu(U_0) = 0 \: .
\]
Then we get
\[
\sum_{j=1}^{2\zeta +1} \mu \left( \bigsqcup_{B \in \cG_\delta^j} B \right) \geq \mu (U) = +\infty \: .
\]
Consequently there exists $j_0 \in \{1, \ldots , 2\zeta +1 \}$ such that $\mu \left( \bigsqcup_{B \in \cG_\delta^{j_0}} B \right) = +\infty$. Therefore we have the same estimate as in the case $\mu (U) < +\infty$:
\begin{align*}
\hmu_\delta^q (U) \geq \sum_{B \in \cG_\delta^{j_0}} q(B) \geq & \sum_l \frac{1}{C} \mu ( B_{\a r_l} (x_l) ) \\
& \geq \frac{1}{C (\g + \epsilon_0)} \sum_l \mu (B_{r_l} (x_l)) = \frac{1}{C(\g + \epsilon_0)} \mu \left( \bigsqcup_{B \in \cG_\delta^{j_0}}  B \right) = + \infty \: ,
\end{align*}
and we conclude $\hmu^{q} (U) = +\infty$.
\end{proof}

\begin{prop} \label{cor_countable_sub_additvity}
Let $(X,d,\nu)$ satisfy Hypothesis \eqref{hypo1}. Let $\mu$ be a positive Borel measure on $X$ and let $q$ be a premeasure satisfying \eqref{eq_mean_premeasure}. Let $\hmu^{q}$ be as in Definition~\ref{dfn_packing_construction}. Then $\hmu^{q}$ is countably sub-additive.
\end{prop}

\begin{proof}
Let $(A_n)_n$ be a countable collection of subsets of $X$. If $\displaystyle \sum_n \mu \left( A_n \right) = +\infty$ then by Proposition~\ref{prop_equivalence_on_open_sets} there exists $K > 0$ such that $\mu (A_n) \leq K \hmu^{q} (A_n)$ for all $n$, therefore
\[
\sum_n \hmu^{q} (A_n) \geq \frac{1}{K} \sum_n \mu (A_n)  = +\infty \: ,
\]
whence the countable sub-additivity directly follows. Recall that if $\displaystyle \sum_n \mu \left(  A_n \right) < +\infty$ and $A_n$ are open sets, then countable sub-additivity was proved in Proposition~\ref{prop_sub_additivity_open_case}. It remains to check the case $\displaystyle \sum_n \mu \left(  A_n \right) < +\infty$ for a generic sequence of Borel sets $A_n$. Fix $\epsilon >0$, then by definition of $\hmu^q$ there exists an open set $U_n$ for all $n$, depending also on $\epsilon$, such that $A_n \subset U_n$ and
\[
\hmu^q (U_n) \leq \hmu^q (A_n) + \frac{\epsilon}{2^n} \: .
\]
By sub-additivity on open sets we have 
\[
\hmu^{q} (\bigcup_n A_n) \leq \hmu^{q} (\bigcup_n U_n) \leq \sum_n \hmu^{q} (U_n) \leq \sum_n \hmu^q (A_n) + \epsilon \: .
\]
Letting $\epsilon \rightarrow 0$ concludes the proof.
\end{proof}




\subsection{Connection with a classical packing construction}
\label{sect:BLversusTT}

Our packing construction \eqref{dfn_packing_construction} is very similar to the one introduced by Taylor and Tricot in \cite{tricot} for measures in $\R^n$. In that paper, starting from a given premeasure $q$ defined on a family of sets $\cC$ (here, as in our construction, $\cC$ will be the family of closed balls) a so-called \emph{packing premeasure} is defined for any $E \subset \R^n$ as
\[
(q-P)(E) = \limsup_{\delta \to 0} \left\lbrace \sum_{B \in \cB} q(B) \: : \: \cB \text{ is a T--packing of order } \delta \text{ of } E, \: \cB \subset \{ B_r (x) \: : \: x \in E, \, r>0 \} \right\rbrace \: ,
\]
where their notion of packing (here specialized to families of closed balls,  which we will refer to as $T$--packing) is slightly different from the one we introduced in Definition~\ref{def:packing}. 
\begin{dfn}[$T$--Packings] \label{def:Tpacking}
Let $E \subset \R^n$. We say that $\cF$ is a $T$-packing of $E$ of order $\delta$ if $\cF$ is a countable family of disjoint closed balls whose radius is less or equal to $\delta$ and such that for all $B \in \cF$, $\overline{E} \cap B \neq \emptyset$.
\end{dfn}
We insist on the fact that such a packing is \emph{not} included in $E$ (as required in Definition ~\ref{def:packing}) but only in some enlargement of $E$.

Then, from this packing premeasure $(q-P)$, a packing measure $\mu^{q-P}$ is defined by applying Carath\'eodory's construction, Method I, to $q-P$ on Borel sets. To be precise, for any $A \subset \R^n$,
\[
\mu^{q-P} (A) = \inf \left\lbrace \sum_{k=1}^\infty (q-P)(A_k) \: : \: A_k \in \cB(\R^n), \: A \subset \bigcup_{k} A_k \right\rbrace \: .
\]
We will now prove that these two constructions are equivalent when the premeasure $q$ is controlled by some Radon measure $\mu$, i.e., when there exists $C >0$ such that for every closed ball $B$, $q(B) \leq C \mu(B)$.

\begin{prop}
Let $q$ be a premeasure defined on closed balls in $\R^{n}$. We assume that there exists a Radon measure $\mu$ and a constant $C>0$ such that $q(B)\le C \mu(B)$ for any closed ball $B\subset \R^{n}$. Then
$\hmu^{q} = \mu^{q-P}$.
\end{prop}

\begin{proof}
Let us prove that for any compact set $K \subset \R^n$
\begin{equation} \label{eq:comparisonTricotCompact0}
\mu^{q-P}(K) \leq \hmu^q (K) \: .
\end{equation}
Fix an open set $U$ containing $K$. As $K$ is compact, there exists $\delta_0 > 0$ such that, for all $0 < \delta < \delta_0$,
\[
K_{2 \delta} = \left\lbrace x \in \R^n \: : \: \dist (x,K) < 2\delta  \right\rbrace \subset U \: .
\]
Note that if $\cB$ is a $T$--packing of order $\delta$ of $K$ then it is a packing of order $\delta$ of $K_{2\delta}$ and thus of $U$. Hence, for all $\delta < \delta_0$
\[
\sup \left\lbrace \sum_{B \in \cB} q(B) \: : \: \cB \ T\text{--packing of } K \text{ of order } \delta  \right\rbrace \leq \sup \left\lbrace \sum_{B \in \cB} q(B) \: : \: \cB \text{ packing of } U \text{ of order } \delta  \right\rbrace \: .
\]
Let us pass to the limit as $\delta \rightarrow 0$ and obtain by definition
\begin{equation} \label{eq:comparisonTricotCompact}
(q-P)(K) \leq \hmu^q(U) \: .
\end{equation}
The inequality \eqref{eq:comparisonTricotCompact} is true for any open set $U \supset K$, therefore
\[
(q-P)(K) \leq \inf_{K \subset U open} \hmu^q (U) = \hmu^q (K)\,.
\]
Since $\{K\}$ is trivially a covering of $K$, then $\mu^{q-P}(K) \leq (q-P)(K)$, which leads to \eqref{eq:comparisonTricotCompact0}. It is easy to extend \eqref{eq:comparisonTricotCompact0} to any bounded Borel set thanks to Theorem~\ref{thm_approximation_Borel_measure} and the fact that $\mu$ is Radon and $\R^n$ is locally compact. Then for any Borel set $E \subset \R^n$,
\[
\mu^{q-P} (E) \leq \hmu^q (E) \: .
\]

Let us now prove that for any bounded set $E \subset \R^n$ we have the converse inequality
\begin{equation} \label{eq:comparisonTricotBounded0}
\hmu^q (E) \leq \mu^{q-P}(E) \: .
\end{equation}
Given $\e>0$ and $0 < \delta < \epsilon/2$ we define
\[
E_\epsilon = \left\lbrace x \in \R^n \: : \: \dist (x,\overline{E}) < \e \right\rbrace \: ,
\]
which is obviously an open set. 
Let $\cB$ be a $\delta$--packing of $E_\e$, then $\cB = \cB_1 \sqcup \cB_2$ with
\[
\cB_1 = \left\lbrace B \in \cB \: : \: B \cap \overline{E} \neq \emptyset \right\rbrace \quad \text{and} \quad \cB_2 = \left\lbrace B \in \cB \: : \: B \cap \overline{E} = \emptyset \right\rbrace \: .
\]
Then, by Definition~\ref{def:Tpacking}, $\cB_1$ is a $T$--packing of order $\delta$ of $E$ while any ball of $\cB_2$ is included in $E_\e - \overline{E}$. Therefore, taking the supremum over all $\delta$--packings of $E_\e$, we get
\[
\hmu^q_\delta (E_\e) \leq \sup \left\lbrace \sum_{\substack{B \in \cB \\ \overline{B} \cap \overline{E} \neq \emptyset}} q(B) \: : \: \cB \, \delta\text{--packing of } E_\e \right\rbrace + \underbrace{ \sup \left\lbrace \sum_{\substack{B \in \cB \\ \overline{B} \cap \overline{E} = \emptyset}} q(B) \: : \: \cB \, \delta\text{--packing of } E_\e \right\rbrace }_{\leq \mu \left( E_\e - \overline{E} \right)}\,,
\]
then letting $\delta \rightarrow 0$ leads to
\begin{equation} \label{eq:comparisonTricotBounded1}
\hmu_\delta^q (E_\e) \leq (q-P)(E) + \mu (E_\e - \overline{E}) \: .
\end{equation}
As $\mu$ is Radon and $E_\e - \overline{E}$ is bounded (since $E$ is bounded), then $\mu (E_\e - \overline{E}) < +\infty$, $( E_\e - \overline{E} )_\epsilon$ is increasing in $\e$ and $\cap_{\e>0}( E_\e - \overline{E} )_\epsilon = \emptyset$, so that
\[
\mu (E_\e - \overline{E}) \xrightarrow[\e \to 0]{} 0 \: .
\]
Letting $\epsilon \rightarrow 0$ in \eqref{eq:comparisonTricotBounded1}, we get
\begin{equation} \label{eq:comparisonTricotBounded2}
\hmu^q (E) = \inf_{E \subset U} \left\lbrace \hmu^q (U) \: : \: U \text{ open } \right\rbrace \leq \lim_{\e \to 0} \hmu^q (E_\e) \leq (q-P)(E) \: .
\end{equation}
Thanks to Proposition~\ref{prop_sub_additivity_loc_finite_case}, we know that $\hmu^q$ is countably sub-additive and for every countable family $(E_h)_h$ such that $E \subset \cup_h E_h$, we have by \eqref{eq:comparisonTricotBounded2}
\[
\hmu^q (E) \leq \sum_h \hmu^q (E \cap E_h) \leq \sum_h (q-P)(E \cap E_h) \leq \sum_h (q - P)(E_h)\,.
\]
Finally taking the infimum over all such $(E_h)_h$ gives \eqref{eq:comparisonTricotBounded0}, whence
\[
\hmu^q(E) \leq \mu^{q-P} (E) 
\]
holds for any Borel set $E \subset \R^n$. This concludes the proof.
\end{proof}

\begin{remk}
Reading carefully the proof, one should note that the property of $\R^n$ which is used is the local compactness of $\R^n$. Therefore, if we extend the definition of Taylor and Tricot \cite{tricot} to a metric space $(X,d)$, then assuming $(X,d)$ locally compact the two packing constructions still coincide under the assumption $q(B) \leq C \mu (B)$ for all closed balls $B$, where $\mu$ is some given Radon measure and $C >0$.
\end{remk}

\subsection{The case of a signed measure}
Our aim is to prove that the packing-type reconstruction applied to a signed measure $\mu$, with premeasures $q_{\pm}$ satisfying 
\begin{equation}\label{cond-signed-pplus}
\frac{1}{C} \mu^+(B_{\a r}(x)) - \mu^{-}(B_{r}(x)) \le q_{+}(B_{r}(x)) \le C\mu^{+}(B_{r}(x))
\end{equation}
and
\begin{equation}\label{cond-signed-pminus}
\frac{1}{C} \mu^-(B_{\a r}(x)) - \mu^{+}(B_{r}(x)) \le q_{-}(B_{r}(x)) \le C\mu^{-}(B_{r}(x))
\end{equation}
for some $C\ge 1$ and $0<\a\le 1$, produces a signed measure $\hmu^{p}$ whose positive and negative parts are comparable with those of $\mu$. 

We notice that properties \eqref{cond-signed-pplus} and \eqref{cond-signed-pminus} are weaker than the following (and apparently more natural) ones: 
\begin{align}
\label{csp1}
\frac{1}{C} \mu^+(B_{\a r}(x)) &\le q_{+}(B_{r}(x)) \le C\mu^{+}(B_{r}(x)) \\ 
\label{csp2}
\frac{1}{C} \mu^-(B_{\a r}(x)) & \le q_{-}(B_{r}(x)) \le C\mu^{-}(B_{r}(x))\,.
\end{align}
On the other hand we note that the premeasures defined as
\[ 
q_{\pm}(B_{r}(x)) =  \left( \frac{1}{r} \int_{s=0}^r  \mu (B_s (x) )  \, ds \right)_\pm 
\]
satisfy \eqref{cond-signed-pplus}--\eqref{cond-signed-pminus} but not \eqref{csp1}--\eqref{csp2}. 


\begin{theo}\label{thm:signedboh}
Let $(X,d,\nu)$ satisfy Hypothesis \ref{hypo1} for some constants $\a\in (0,1]$ and $\g\ge 1$, and let $\mu = \mu^+ - \mu^-$ be a locally finite, Borel-regular signed measure on $X$. Let $q_\pm$ be a pair of premeasures satisfying \eqref{cond-signed-pplus} and \eqref{cond-signed-pminus} for some $C\ge 1$. Take $\hmu^{q_\pm}$ as in Definition~\ref{dfn_packing_construction}. Then the following properties hold:
\begin{itemize}
\item[(i)] $\hmu^{q_+}$, $\hmu^{q_-}$ are metric outer measures finite on bounded sets;
\item[(ii)] $\hmu^q = \hmu^{q_+} - \hmu^{q_{-}}$ is a signed measure such that, for any Borel set $A \subset X$,
\[
\frac{1}{\g C} \mu^+ (A) \leq  \hmu^{q_+} (A) \leq  C\mu^+ (A)\quad \text{ and }\quad \frac{1}{\g C} \mu^- (A) \leq  \hmu^{q_-} (A) \leq  C\mu^- (A) \: ,
\]
whence in particular
\[
\frac{1}{\g C} | \mu | (A) \leq  |\hmu^q | (A)  \leq  C | \mu | (A) \: .
\]
\end{itemize}
\end{theo}

\begin{proof}
The countable sub-additivity of $\hmu^{q_{\pm}}$ follows from the second inequalities in \eqref{cond-signed-pplus}--\eqref{cond-signed-pminus} (see Proposition~\ref{prop_sub_additivity_loc_finite_case}). 
Then for any open set $U \subset X$ both inequalities
\[
\hmu^{q_\pm} (U) \leq C\mu^\pm (U) 
\]
are just a consequence of the definition of $\hmu^{q_{\pm}}$ and of the second inequalities in \eqref{cond-signed-pplus}--\eqref{cond-signed-pminus}.
This proves (i). 

Let now $A \subset X$ be a Borel set. We first derive an estimate concerning $\hmu^{q_{+}}$ (the estimate for $\hmu^{q_{-}}$ can be obtained in the same way). If $\mu^+ (A) < +\infty$, we take an open set $U$ containing $A$ such that $\mu^{+}(U)<+\infty$. Let $\e_{0},\delta > 0$ be fixed, then apply Corollary~\ref{bescivoic_with_doubling_balls} to $\mu^{+}$ and get a countable family $\cG_\delta = \{B_{r_{j}}(x_{j})\}_{j}$ of disjoint closed balls contained in $U$ with radii $r_{j}\leq \delta$ and satisfying $\mu^+ (B_{\a r_{j}}(x_{j})) \geq \frac{1}{\g +\e_{0}} \mu^+ (B_{r_{j}}(x_{j}))$ for all $j$, such that
\[
\mu^+ (A) = \mu^+ \left(\bigsqcup_{B \in \cG_\delta } B \right) \: . 
\]
We have
\begin{align*}
\hmu^{q_+}_\delta (U) & \geq \sum_{B \in \cG_\delta} q_+ (B) \geq \sum_{j} \frac 1C \mu^{+}(B_{\a r_{j}}(x_{j})) - \mu^{-}(B_{r_{j}}(x_{j}))\\ 
& \geq \sum_{j} \frac 1{(\g + \e_{0})C} \mu^{+}(B_{r_{j}}(x_{j})) - \mu^{-}(B_{r_{j}}(x_{j})) \geq \frac 1{(\g + \e_{0})C} \mu^{+}(A) - \mu^{-}(U)\,.
%
%
\end{align*}
Letting $\delta\to 0$ and then $\e_{0}\to 0$ we find
\begin{equation} \label{eq_signed_case_1}
\hmu^{q_+} (U)  \geq \frac 1{\g C} \mu^+ (A) - \mu^- (U) \: .
\end{equation}
By definition of $\hmu^{q_+}(A)$, there exists a sequence of open sets $(U_k^1)_k$ such that, for all $k$, it holds $A \subset U_k^1$ and
\[
\hmu^{q_+} (U_k^1) \xrightarrow[k \to \infty]{} \hmu^{q_+} (A) \: .
\]
By outer regularity of $\mu^-$ (which is Borel and finite on bounded sets) there exists a sequence of open sets $(U_k^2)_k$ such that, for all $k$, we get $A \subset U_k^2$ and
\[
\mu^- (U_k^2) \xrightarrow[k \to \infty]{} \mu^- (A) \: .
\]
For all $k$, let $U_k = U_k^1 \cap U_k^2$, then $U_k$ is an open set, $A\subset U_k$ and, by monotonicity,
\begin{align*}
\hmu^{q_+} (A) & \leq \hmu^{q_+} (U_k) \leq \hmu^{q_+} (U_k^1)\,, \\
\mu^- (A) & \leq \mu^- (U_k) \leq \mu^- (U_k^2) \: ,
\end{align*}
therefore
\[
\hmu^{q_+} (U_k) \xrightarrow[k \to \infty]{} \hmu^{q_+} (A) \text{ and } \mu^- (U_k) \xrightarrow[k \to \infty]{} \mu^-(A) \: .
\]
Evaluating \eqref{eq_signed_case_1} on the sequence $(U_k)_k$ and letting $k$ go to $+\infty$, we eventually get
\begin{equation} \label{eq_signed_case_2}
\hmu^{q_+} (A)  \geq \frac{1}{\g C} \mu^+ (A) - \mu^- (A) \: .
\end{equation}
Owing to Hahn decomposition of signed measures, there exists a Borel set $P$ such that for all Borel $A$ it holds
\[
\mu^+ (A) = \mu^+ (A \cap P) =  \mu (A \cap P) \text{ and } \mu^- (A) = \mu (A - P)\: .
\]
Finally, let $A$ be a Borel set, then by \eqref{eq_signed_case_2} applied to $A \cap P$ we get
\begin{align*}
\hmu^{q_+} (A) \geq \hmu^{q_+} (A \cap P) & \geq \frac 1{\g C} \mu^+ (A \cap P) - \mu^- (A \cap P) \\
& = \frac 1{\g C} \mu^+ (A)\,.
\end{align*}
It remains to show that if $\mu^+ (A) = +\infty$, then $\hmu^{q_+} (A) = +\infty$. This can be easily obtained by taking a sequence of open balls $U_{n}$ with fixed center and radius $n\in \N$, and by considering the sequence $A_{n} = A\cap U_{n}$ for which $\mu^{+}(A_{n})<+\infty$ and $\lim_{n} \mu^{+}(A_{n}) = \mu^{+}(A) = +\infty$. By applying the same argument as before, we get 
\[
\hmu^{q_{+}}(A)\ge  \hmu^{q_{+}}(A_{n}) \ge \frac 1{\g C}\mu^{+}(A_{n})\,,
\]
thus the conclusion follows by taking the limit as $n\to +\infty$.
This completes the proof of (ii) and thus of the theorem.
\end{proof}

\subsection{A stability result}
If the approximate values $q(B_{r}(x))$ are closer and closer to the actual values of $\mu(B_{r}(x))$ when $r\to 0$ one obtains by Theorem \ref{thm_main} that the reconstructed measure $\hmu^{q}$ coincides with $\mu$. More precisely, we have the following stability result.
\begin{corollary}
Let us fix $(\a_{n})_{n}, (\g_{n})_{n}, (C_{n})_{n}$ and $(r_{n})_{n}$ such that 
\begin{align*}
&\a_{n}\in (0,1],\ \g_{n}\ge 1,\ C_{n}\ge 1,\ r_{n}>0 \\ 
&\a_{n},\g_{n},C_{n}\to 1\quad\text{and}\quad r_{n}\to 0\quad \text{as }n\to \infty\,.
\end{align*}
Let $(X,d)$ be a directionally limited metric space endowed with a sequence of asymptotically $(\a_{n},\g_{n})$-bounded measures $\nu_{n}$ satisfying \eqref{agbounded}. Let $\mu$ be a positive Borel measure on $X$ and let $q$ be a premeasure defined on closed balls and satisfying 
\[
C_{n}^{-1}\mu(B_{\a_{n}r}(x)) \le q(B_{r}(x)) \le C_{n}\, \mu(B_{r}(x))
\] 
for all $x\in X$, $n\in \N$, and $r\in (0,r_{n})$. Then $\hmu^{q} = \mu$.
\end{corollary}
\begin{proof}
It is an immediate consequence of Theorem \ref{thm_main}.
\end{proof}

\begin{remk}
The above corollary can be formulated for signed measures as well. Indeed under the same assumptions on $X$ and analogous assumptions on $q_{\pm}$ one obtains the same conclusion, i.e. the coincidence of the reconstructed signed measure with the initial measure, thanks to Theorem \ref{thm:signedboh}.
\end{remk}

\section*{Acknowledgements}
\addcontentsline{toc}{section}{Acknowledgements}
This work has been co-funded by GNAMPA (INdAM) and by PALSE program. We also thank Carlo Benassi for fruitful discussions on the subject.

\end{document}